\documentclass[12pt]{amsart}
\pdfoutput=1
\usepackage[cal=rsfso,calscaled=.96]{mathalfa}

\makeatletter
\def\subsection{\@startsection{section}{1}%
  \z@{1.2\linespacing\@plus\linespacing}{.5\linespacing}%
  {\normalfont\scshape\centering}}
\def\subsection{\@startsection{subsection}{2}%
  \z@{0.9\linespacing\@plus.7\linespacing}{-.5em}%
  {\normalfont\bfseries}}
\makeatother

\usepackage{amsmath}
\usepackage{amsfonts}
\usepackage{amsbsy}
\usepackage{amssymb}
\usepackage[normalem]{ulem}
\usepackage{times}
\usepackage{mathrsfs}
\usepackage{exscale}
\usepackage{graphicx}
\usepackage[colorlinks,citecolor=blue,linkcolor=rouge,
            bookmarksopen,
            bookmarksnumbered
           ]{hyperref}
\usepackage{dsfont}
\usepackage{times}

\usepackage{color}
\definecolor{gr}{rgb}   {0.,   0.6,   0.25 }
\definecolor{mg}{rgb}   {0.85,  0.,    0.85}
\definecolor{marin}{rgb}   {0.,   0.,   0.8}
\definecolor{rouge}{rgb}   {0.8,   0.,   0.}
\definecolor{orange}{rgb}   {0.8,   0.4,   0.}

\usepackage{booktabs}

\newtheorem{theorem}{Theorem}[section]
\newtheorem{lemma}[theorem]{Lemma}
\newtheorem{proposition}[theorem]{Proposition}
\newtheorem{corollary}[theorem]{Corollary}
\newtheorem{conjecture}[theorem]{Conjecture}

\theoremstyle{definition}

\theoremstyle{remark}
\newtheorem{remark}[theorem]{Remark}

\numberwithin{equation}{section}

\newcommand{\dd}[1]{_{\raise-0.3ex\hbox{$\scriptstyle #1$}}}

\newcommand{\on}[1]{\raise-.5ex\hbox{\big|}_{#1}}

\newcommand{\pv}{\mathop{\rm p.v.}}
\newcommand{\im}{\operatorname{\rm im}}

\renewcommand{\div}{\operatorname{\rm div}}

\newcommand{\curl}{\operatorname{\rm curl}}

\newcommand\C{{\mathbb C}}
\newcommand\R{{\mathbb R}}
\newcommand\N{{\mathbb N}}

\renewcommand\SS{{\mathbb S}}

\newcommand\Z{{\mathbb Z}}

\newcommand{\sF}{{\mathscr F}}

\newcommand{\sS}{{\mathscr S}}

\newcommand\cC{{\mathcal{C}}}

\newcommand\cL{{\mathcal{L}}}

\newcommand {\Id}{\mathbb I}

\newcommand{\dist}{\mathrm{dist}}

\newcommand{\sign}{\mathrm{sign}}
\newcommand{\Sp}{\mathop{\mathrm{Sp}}\nolimits}

\usepackage{soul}

\let\hat=\widehat

\title[Stability of a Simple Discretization Method]{
Stability Analysis of a Simple Discretization Method\\ for a Class of \\Strongly Singular Integral Equations}
\author{Martin Costabel}
\author{Monique Dauge}
\author{Khadijeh Nedaiasl}
\address{Univ. Rennes, CNRS, IRMAR - UMR 6625, F-35000 Rennes, France}
\email{Martin.Costabel@univ-rennes1.fr}
\email{Monique.Dauge@univ-rennes1.fr}
\address{Department of Mathematics, IASBS, Gavazang Road, Zanjan, Iran}
\email{knedaiasl85@gmail.com, nedaiasl@iasbs.ac.ir}
\keywords{volume integral equation, strongly singular kernel, delta-delta discretization, discrete dipole approximation, numerical stability}
\subjclass{65R20, 45E10, 47A12, 65B10, 78M99}
\begin{document}

\begin{abstract}
Motivated by the discrete dipole approximation (DDA) for the scattering of electromagnetic waves by a dielectric obstacle that can be considered as a simple discretization of a Lippmann-Schwinger style volume integral equation for time-harmonic Maxwell equations, we analyze an 
analogous discretization of convolution operators with strongly singular kernels. 

For a class of kernel functions that includes the finite Hilbert transformation in 1D and the principal part of the Maxwell volume integral operator used for DDA in dimensions 2 and 3, we show that the method, which does not fit into known frameworks of projection methods, can nevertheless be considered as a finite section method for an infinite block Toeplitz matrix. The symbol of this matrix is given by a Fourier series that does not converge absolutely. We use Ewald's method to obtain an exponentially fast convergent series representation of this symbol and show that it is a bounded function, thereby allowing to describe the spectrum and the numerical range of the matrix.

It turns out that this numerical range includes the numerical range of the integral operator, but that it is in some cases strictly larger. In these cases the discretization method does not provide a spectrally correct approximation, and while it is stable for a large range of the spectral parameter $\lambda$, there are values of $\lambda$ for which the singular integral equation is well posed, but the discretization method is unstable.

\end{abstract}

\maketitle

\tableofcontents

%%%%%
\section{Introduction}\label{S:intro}
%%%%%

%%%%%%%%%
\subsection{Motivation}\label{SS:motiv}\ 
%%%%%%%%%

Introduced almost 50 years ago by Purcell and Pennypacker \cite{PurcellPennypacker1973}, the Discrete Dipole Approximation (DDA) is a classical numerical method in computational electromagnetics that is the subject of a vast and still rapidly growing literature (see the surveys \cite{YurkinHoekstra2007,Chaumet_M_22}), but is virtually unknown in the mathematical community. It can be considered as a numerical approximation scheme for a strongly singular volume integral equation that, however, is too simple to fit into any known framework for standard approximation schemes for such equations (Galerkin, collocation or Nystr\"om methods etc). In particular, to the authors' knowledge, there does not exist any error estimate or convergence proof for this method. 

In the paper  \cite{Yurkin:06i}, estimates for a consistency error are derived, and it is observed that to complete the convergence analysis, a uniform estimate for the inverse of the matrix of the linear system (stability estimate) would be needed. In the present paper, we prove first results on the way to such stability estimates for the DDA and related numerical schemes. The class of singular integral equations considered here includes the quasi-static case (i.e.\ zero frequency limit) of the Maxwell volume integral equation that describes the scattering of electromagnetic waves by a penetrable dielectric body in the case of constant electric permittivity. Further stability results for the non-zero frequency case will be the subject of a forthcoming paper.

Because of the simplicity of the class of operators considered here (convolution operators with kernels positively homogeneous of degree $-d$ on a bounded domain $\Omega\subset\R^{d}$), we are able to obtain rather sharp results on the region of stability, by estimating the numerical range of the discretized operator in comparison with the numerical range of the integral operator. It turns out that for some operators, including the quasi-static Maxwell case in dimension $d\ge2$, the stability region is smaller than what one would na\"ively expect. 
This corresponds to the fact that the eigenvalues of the system matrix, as the mesh-width of the discretization tends to zero, accumulate on a set that is strictly larger than the convex hull of the essential spectrum of the integral operator.

In the paper \cite{Rahola2000}, motivated by the convergence analysis of iterative solutions of the resulting large linear systems, the essential spectrum of the Maxwell volume integral operator was studied for the case of scattering by a dielectric ball in $\R^{3}$. This is a subset of the segment in the complex plane that corresponds to the essential numerical range of the integral operator. It is now known (see \cite{CoDarSak2012,CoDarSak2015}) that the same form of the essential spectrum is valid for more general bounded Lipschitz domains.
In \cite{Rahola2000}, results of some numerical experiments are then shown that seem to indicate that the eigenvalues of the system matrices accumulate either at isolated points, corresponding to eigenvalues of the integral operator and hence to eigenvalues or resonances of the scattering problem, or at the points of the segment that is spanned by the essential spectrum of the integral operator. Looking closer at Figures 4.2--4.4 of \cite{Rahola2000}, one can detect an ``overshoot'', namely that the observed segment of accumulation points is actually larger than the span of the essential spectrum. 

In the paper \cite{YurkinMinHoekstra2010}, there is a discussion of the spectrum of the system matrices of the DDA scheme for the quasi-static Maxwell equations, motivated by the numerical modeling of the scattering of light by dust particles whose size is small with respect to the wavelength of the light (``Rayleigh particles''). Based on extensive experience with numerical computations using the DDA code ADDA, the authors are convinced that the DDA provides a faithful approximation of the solution of the volume integral equation in the sense that, among other things, the spectral measure of the DDA system matrices converges to the spectral measure of the volume integral operator. They provide plots of the spectral density of these matrices, including a zoom on a neighborhood of the lower end of the spectrum, see graph (a) in \cite[FIG. 8]{YurkinMinHoekstra2010}. There one can clearly see that there is an overshoot, namely a part of the spectrum below zero, and that its negative minimum does not disappear as the number of dipoles grows, but rather seems to converge to some number around $-0.09$.
In a subsequent paper \cite{Smunev_JQSRT_15}, the authors detect this ``spill-out'' of the spectrum of the DDA system matrices and relate it to an explosion of the needed iterations in an iterative solution method that they observed for large refractive indices. They study the behavior of this overshoot for anisotropic meshes, where it becomes larger, and for some recently introduced improvements of the DDA, where it seems to disappear.

For the quasi-static Maxwell case we prove below (see Proposition~\ref{P:W(F)Ex5} and \eqref{E:Lambda+-3}) that for the classical DDA on a cubic grid such an overshoot indeed exists and that it amounts to an almost $20\%$ increase of the length of the segment spanned by the essential spectrum. 

This somewhat unexpected result implies that the simple discretization scheme of the DDA does not provide a spectrally correct approximation of the strongly singular volume integral operator. The additional observation, supported by numerical experience, that this concerns only a small neighborhood of the essential spectrum or perhaps even only of the endpoints of this spectrum, whereas discrete eigenvalues and large parts of the spectral density nevertheless are correctly approximated, still awaits a precise description and proof.

It also implies that the DDA scheme is actually unstable in high-contrast situations, namely if the relative permittivity is very small (smaller than $\sim 0.093$ ) or very large (larger than $\sim 11.8$). We prove this here for the zero-frequency limit, but expect that it is also true for non-zero frequencies.

%%%%%%%%%
\subsection{The Discrete Dipole Approximation}\label{SS:DDA}\ 
%%%%%%%%%

As its name indicates, the DDA (sometimes called Coupled Dipole Approximation) can be considered as an approximation of a dielectric continuum described by Maxwell's equations by a different physical system consisting of a finite number of dipoles that are characterized by their polarizability, interacting via electromagnetic fields.  

The same mathematical system can be obtained by a procedure more amenable to arguments of numerical analysis, namely by transforming the Maxwell equations for the original dielectric continuum into an equivalent Lippmann-Schwinger style volume integral equation and then discretizing this integral equation by a simple delta-delta approximation on a regular grid
$\{x_{n}\mid n\in\Z^{d}\}$ of meshwidth $h>0$.

Thus a linear integral equation on a bounded domain $\Omega\subset\R^{d}$
\begin{equation}
\label{E:IE}
  \lambda u(x) - \int_{\Omega}K(x,y)u(y)dy = f(x) \quad(x\in\Omega)
\end{equation}
will be approximated by the finite dimensional linear system 
\begin{equation}
\label{E:DDAp}
  \lambda u_{m} - \sum_{x_{n}\in\Omega,n\ne m}h^{d}K(x_{m},x_{n})u_{n} = f(x_{m})
  \quad(x_{m}\in\Omega)\,.
\end{equation}
We omit the diagonal term $m=n$, because we shall have to do with singular kernels. Apart from this, \eqref{E:DDAp} looks like a Galerkin method with Dirac deltas as trial and test functions.

Let us briefly describe the construction of the volume integral equation.
A more detailed derivation can be found in \cite{Kirsch2007} and, with special emphasis on the two-dimensional situation, in \cite{CoDarSak2015}.  
We write the time harmonic Maxwell equations with normalized frequency $\kappa\in\C$ as a second order system for the electric field $u$.
\begin{equation}
\label{E:Maxwell}
\curl\curl u -\kappa^{2}\epsilon u= i\kappa J.
\end{equation}
Here it is assumed that the magnetic permeability is constant (normalized to $1$) in the whole space. If one further assumes that the permittivity $\epsilon$ is equal to $1$ outside of a bounded domain and the source current $J$ has compact support, one can write this as a perturbation of the free-space situation
\begin{equation}
\label{E:Maxwellpert}
\curl\curl u -\kappa^{2} u= -\kappa^{2}(1-\epsilon) u + i\kappa J.
\end{equation}
Here the right hand side has compact support, and therefore convolution with the outgoing fundamental solution $g_{\kappa}$ of the Helmholtz equation and application of the operator 
$\nabla\div + \kappa^{2}$ leads to the volume integral equation in distributional form
\begin{equation}
\label{E:VIEdist}
 u = -(\nabla\div + \kappa^{2})g_{\kappa}\star(1-\epsilon)u + u^{\rm inc} \,.
\end{equation}
Here the incoming field $u^{\rm inc}$ combines the field generated by the current density with possible sourceless full space solutions of Maxwell's equations (plane waves etc.) 

Equation \eqref{E:VIEdist}  can be considered in any dimension $d\ge2$, but only $d=2$ and $d=3$ are relevant for electrodynamics. The equation can be written in the form  of a second kind strongly singular integral equation with the $d\times d$ matrix valued kernel
\begin{equation}
\label{E:VIOkernel}
 K(x,y) = -(D^{2}+\kappa^{2})g_{\kappa}(x-y)\,.
\end{equation}
The integral operator thus defined involves second order distributional derivatives of the weakly singular kernel $g_{\kappa}(x-y)$. Instead of this form of an integro-differential operator, one can write the strongly singular integral operator also in the form of a Cauchy principal value integral, using the well-known relation (for more details, see section~\ref{SSS:SIOEx5} below)
\begin{equation}
\label{E:D2g-pv}
  D^{2}\int_{\R^{d}}g_{\kappa}(x-y)u(y)dy = 
  \pv\!\!\int_{\R^{d}} D^{2}g_{\kappa}(x-y)u(y)dy  - \frac1d u(x)\,.
\end{equation}
If we further assume that the permittivity $\epsilon$ equals a constant $\epsilon_{r}\in\C\setminus\{1\}$ in $\Omega$, we can divide by $1-\epsilon_{r}$ and arrive at the final form \eqref{E:IE} with the integral understood in the principal value sense, the kernel given by \eqref{E:VIOkernel}, and the spectral parameter $\lambda$ defined by the relation
\begin{equation}
\label{E:ClauMoss}
 \lambda = \frac1{1-\epsilon_{r}} - \frac1d = \frac{d-1+\epsilon_{r}}{d(1-\epsilon_{r})}\,.
\end{equation}
For $d=3$, this relation $\lambda=\frac{2+\epsilon_{r}}{3(1-\epsilon_{r})}$ is known in the DDA literature as \emph{Clausius-Mossotti polarizability}, referring to the fact that $\frac1\lambda$ corresponds to the polarizability of the dipoles and to the Clausius-Mossotti equation between the molecular polarizability and the electric permittivity in a dielectric material, see for example \cite[Section 4.5]{Jackson1999}. 

The principal part of the volume integral operator is obtained by taking the limit $\kappa\to0$, and we will refer to this situation as the quasi-static Maxwell case. 
The resulting kernel is homogeneous of degree $-d$, and this property allows to analyze the corresponding linear system \eqref{E:DDAp} using Fourier analysis of Toeplitz matrices. For this reason we study in this paper a class of strongly singular kernels that includes the quasi-static Maxwell kernel.

%%%%%%%%%
\subsection{Outline of the paper}\label{SS:outline}\ 
%%%%%%%%%

In Section~\ref{SS:kernels} we define a class of strongly singular kernels that are homogeneous of degree $-d$ and translation invariant, and we evoke the relation between the numerical range of the corresponding singular integral operator in $L^{2}$ and values of its symbol. The notion of numerical range allows to use a Lax-Milgram type argument to get a resolvent estimate for the restriction of the convolution operator to a bounded domain $\Omega$.

After introducing in Section~\ref{SS:DDD} the delta-delta discretization, we state in Theorem~\ref{T:stabgeneral} the main stability result valid for our class of operators. 

In Section~\ref{S:discrete} we study tools for proving stability results, namely infinite Toeplitz matrices and their symbols defined by Fourier series. Here a main difficulty is that one needs precise bounds for the values of a function (numerical symbol) defined by a Fourier series that is not absolutely convergent. We find that one can use Ewald's method for this purpose.  
The result is that the symbol of the Toeplitz matrix is a bounded function, and that its range is always a superset of the range of the symbol of the integral operator, but that it might be strictly larger.
If this is the case, then stability of the delta-delta scheme implies well-posedness of the integral equation, but not vice versa: The numerical scheme does then not provide a spectrally correct approximation, and it might be unstable for values of the spectral parameter $\lambda$ for which any Galerkin scheme of the integral equation would be stable.

In Section~\ref{S:Examples} we study in detail five representative examples. 

Example \ref{Ex:1} concerns the one-dimensional singular integral equation defined by the finite Hilbert transformation. Here the numerical symbol has a simple explicit expression, and this can be used to get estimates for the resolvent of the discretized operator by the resolvent of the integral operator, with constant equal to $1$. This gives Theorem~\ref{T:Stab1D}, which is the ideal stability result that subsequent results are measured against. 

Examples \ref{Ex:2} and \ref{Ex:3} exhibit different behavior of the delta-delta scheme for two strongly singular integral operators in two dimensions. Whereas the two integral operators are equivalent, related by a simple rotation of the coordinate system, the two discrete systems show opposite behavior: We prove that in Example \ref{Ex:2} the ranges of the symbol of the integral operator and of the symbol of the infinite Toeplitz matrix are identical, whereas in Example \ref{Ex:3} there is an overshoot; the region of instability of the approximation scheme is strictly larger than the numerical range of the integral operator. 

In Example \ref{Ex:4}, we graphically illustrate the relations, proved in Sections~\ref{SS:kernels} and \ref{S:discrete},
 between the spectrum and numerical range of the system matrices and the numerical range of the singular integral operator by considering a non-selfadjoint case. The kernel is a complex-valued function whose real and imaginary parts are given by the kernels of Examples \ref{Ex:3} and \ref{Ex:2}, respectively. 
 
The kernels studied in Examples \ref{Ex:2} and \ref{Ex:3} are also the off-diagonal and diagonal terms, respectively, in the matrix-valued kernel of the quasi-static Maxwell volume integral operator, which is the subject of Example \ref{Ex:5}. We study this for dimensions $d\ge2$ and give more precise results for $d=2$ and $d=3$. In two dimensions we find the same overshoot of the numerical range of the numerical symbol versus the symbol of the integral operator as Example \ref{Ex:3}. In three dimensions this overshoot is even larger, and it can be verified numerically either by computing the numerical symbol using Ewald's method or by studying the asymptotic behavior of the smallest and largest eigenvalues of the matrix of the linear system as the mesh width tends to zero.

%%%%%%%%%
\subsection{Notation for Fourier transforms and Fourier series}\label{SS:notation}\ 
%%%%%%%%%

We use the following convention for the Fourier transformation in $\R^{d}$.
\begin{equation}
\label{E:FT}
\hat{f}(\xi)=\sF f(\xi)=\int_{\R^{d}}f(x)e^{i\xi\cdot x}dx\,.
\end{equation}
Inverse: 
\begin{equation}
\label{E:FTinv} 
f(x) = \sF^{-1}\hat{f}(x)=(2\pi)^{-d}\int_{\R^{d}}\hat f(\xi)e^{-ix\cdot \xi}d\xi\,.
\end{equation}
For Fourier series, we use the following notation.
For a sequence $a:\Z^{d}\to\C$, its Fourier series is defined as
\begin{equation}
\label{E:FS}
\tilde a(\tau) = \sum_{m\in\Z^{d}}a(m) e^{im\cdot \tau}\,,\quad
\tau \in Q=[-\pi,\pi]^{d} .
\end{equation}
Inverse:
\begin{equation}
\label{E:FSinv}
a(m) = (2\pi)^{-d}\int_{Q}\tilde a(\tau) e^{-im\cdot \tau}d\tau.
\end{equation}
The definitions are extended in the usual way from convergent sums and integrals to suitable spaces of functions and distributions. In particular, we have Parseval's theorem
\begin{equation}
\label{E:Parseval}
f\mapsto (2\pi)^{-\frac d2}\hat{f}: L^{2}(\R^{d}) \to L^{2}(\R^{d}) 
\quad\mbox{ and }\quad
a\mapsto (2\pi)^{-\frac d2}\tilde{a}: \ell^{2}(\Z^{d}) \to L^{2}(Q) 
\end{equation}
are unitary (i.e.\ isometric Hilbert space isomorphisms).

Combining the Parseval formula and the convolution theorem gives
\begin{equation}
\label{E:Par+conv,int}
\int_{\R^{2d}}\overline{u(x)}\,k(x-y)\,v(y)\,dy\,dx
= (2\pi)^{-d}\int_{\R^{d}}\overline{\hat u(\xi)}\,\hat k(\xi)\,\hat v(\xi)\,d\xi\,,
\end{equation}
\begin{equation}
\label{E:Par+conv,series}
\sum_{m.n\in\Z^{d}}\overline{a(m)}\,c(m-n)\,b(n) = 
(2\pi)^{-d} \int_{Q}\overline{\tilde{a}(\tau)}\,\tilde c(\tau) \,\tilde b(\tau)\, d\tau\,.
\end{equation}
From these formulas follows immediately that the operators of convolution with $k$ in $L^{2}(\R^{d})$ and of discrete convolution with $c$ in $\ell^{2}(\Z^{d})$ are bounded if and only if the ``symbols'' $\hat k$ and $\tilde c$ are bounded functions belonging to $L^{\infty}(\R^{d})$ and $L^{\infty}(Q)$, respectively. 

Sufficient conditions for this are that $k\in L^{1}(\R^{d})$ and $c\in\ell^{1}(\Z^{d})$. But these conditions are not necessary, and it is precisely the situation where they are not satisfied that will be relevant in the following.

We will use the \emph{Poisson summation formula} in the form
\begin{equation}
\label{E:Pois}
\sum_{m\in\Z^{d}} f(m) e^{im\cdot \tau} =
\sum_{n\in\Z^{d}} \hat f(\tau+2\pi n)\,.
\end{equation}
A sufficient (but in no way necessary) condition for \eqref{E:Pois} to hold for all $\tau$ is that 
$$
 f\on{\Z^{d}}\in\ell^{1}(\Z^{d})\quad \mbox{ and }\; \hat f\in L^{1}(\R^{d}).
$$ 

If we do not assume $f\on{\Z^{d}}\in\ell^{1}$, but only $\hat f\in L^{1}$, then $f$ is bounded, the left hand side of \eqref{E:Pois} converges in the distributional sense and the right hand side converges in $L^{1}(Q)$. Then
\eqref{E:Pois} is true in a weaker sense, the distributional left hand side being equal to the $L^{1}(Q)$ right hand side. 

\emph{Example}: Gaussian with parameter $s>0$.
\begin{equation}
\label{E:FGauss}
 f(x) = e^{-|x|^{2}s} \quad\Longleftrightarrow\quad
 \hat f(\xi)= (\tfrac\pi s)^{\frac d2} \,e^{-\frac{|\xi|^{2}}{4s}}\,.
\end{equation}
For this example the Poisson summation formula takes the form (for $\tau\in\R^{d}$)
\begin{equation}
\label{E:PGauss}
\sum_{m\in\Z^{d}} e^{-|m|^{2}s} e^{im\cdot \tau} =
\sum_{n\in\Z^{d}} (\tfrac\pi s)^{\frac d2} \,e^{-\frac{|\tau+2\pi n|^{2}}{4s}}\,.
\end{equation}

%%%%%%%%%
\subsection{Kernels and their symbols}\label{SS:kernels}
%%%%%%%%%
%%%%%%%%%
\subsubsection{Homogeneous kernels}\label{SSS:Homog}
%%%%%%%%%
Later on, we will consider a rather restricted class of strongly singular integral operators on $\R^{d}$ that are convolutions with 
kernel functions of the form 
\begin{equation}
\label{E:Kernels}
 K(x) = p(x)\,|x|^{-d-2} \quad \mbox{where $p$ is a homogeneous polynomial of degree $2$}\,.
\end{equation}

But first we recall some well-known general properties of homogeneous functions and distributions that can be found, for example, in \cite[Chap. III]{Gelfand-Shilov1_1964}.

Let $K$ be a function on $\R^{d}$, positively homogeneous of degree $-d$ and smooth outside of the origin. 
For a given $\epsilon>0$, one can define a distribution $K_{\epsilon}\in\sS'(\R^{d})$ that coincides with $K$ on $\R^{d}\setminus\{0\}$ by its action on a test function $\phi$ as
\begin{equation}
\label{E:Kepsphi}
 \langle K_{\epsilon}, \phi \rangle = \int_{|x|<\epsilon}K(x)(\phi(x)-\phi(0))dx
  + \int_{|x|>\epsilon} K(x)\phi(x)\,dx\,.
\end{equation}
This is independent of $\epsilon$ if and only if $K$ satisfies the cancellation condition on the unit sphere $\SS^{d-1}$
\begin{equation}
\label{E:Kvanish}
\int_{\SS^{d-1}}K\, ds=0\,.
\end{equation}
In this case, we denote the distribution simply by $K$, and we can take the limit 
$\epsilon\to0$, thus we get the Cauchy principal value.
\begin{equation}
\label{E:pvK}
 \langle K, \phi \rangle =
  \pv\!\!\int K(x)\phi(x)\, dx =
  \lim_{\epsilon\to0}\int_{|x|>\epsilon} K(x)\phi(x)\, dx\,.
\end{equation}
Another consequence of the cancellation condition \eqref{E:Kvanish} is that the Fourier transform $\hat K$ of the homogeneous distribution $K$ is a bounded function homogeneous of degree $0$, smooth outside of the origin and also satisfying the cancellation condition. The operator $A$ of convolution with $K$ is therefore bounded in $L^{2}(\R^{d})$.
Note that in the absence of condition \eqref{E:Kvanish}, $\hat K_{\epsilon}$ would have a logarithmic singularity at $0$.

The operator $A$ is diagonalized by Fourier transformation:
\begin{equation}
\label{E:FMul}
\sF Au = \hat K \,\hat u\,.
\end{equation}

Therefore in $L^{2}(\R^{d})$, we can obtain information about the spectrum $\Sp(A)$ and about the numerical range $W(A)$  from the corresponding easily checked information about the operator of multiplication by the symbol $\hat K$. 

We recall that the numerical range of $A$ is defined by
 $$
  W(A) = \{(u,Au) \mid \|u\|=1\}\,,
$$
where $(\cdot,\cdot)$ denotes the Hilbert space inner product. It is convex by the Toeplitz-Hausdorff theorem and it contains the spectrum of $A$.
Denote by $\im(\hat K)=\hat K(\R^{d})$ the image (range)  of $\hat K$. This is a compact set.
We note a first result implied by the unitary equivalence \eqref{E:FMul} with the multiplication operator.
\begin{lemma}
 \label{L:Sp+W(A)}
 The spectrum $\Sp(A)$ is the image $\im(\hat K)$, and the closure $\overline{W(A)}$ of the numerical range of $A$ is the convex hull of $\,\im(\hat K)$.
\end{lemma}

It is well known (and easy to prove) that the numerical range allows estimates for the operator norm of the resolvent: For any 
$\lambda\in\C\setminus\overline{W(A)}$,
$$
 \|(\lambda\Id-A)^{-1}\| \le \dist(\lambda,W(A))^{-1}\,.
$$
It is also monotone with respect to inclusions of subspaces, a property not shared by the spectrum. Given an open set $\Omega\subset\R^{d}$, we denote by $A_{\Omega}$ the restriction of the convolution operator $A$ to $L^{2}(\Omega)$ and consider the strongly singular integral equation $(\lambda \Id -A_{\Omega})u=f$, or in detail
\begin{equation}
\label{E:SIE}
\lambda u(x) - \pv\!\! \int_{\Omega}K(x-y) u(y)\,dy = f(x)\quad (x\in\Omega).
\end{equation}
From the definition of the numerical range follows immediately the inclusion
$W(A_{\Omega})\subset W(A)$. 

We can summarize this discussion:
\begin{proposition}
 \label{P:W(AOmega)}
 Let $\cC\subset\C$ be a closed convex set such that 
 $\hat K(\xi)\in \cC$ for all $\xi\in \SS^{d-1}$. Then for all $\lambda\not\in\cC$ and any $f\in L^{2}(\Omega)$, the integral equation \eqref{E:SIE} has a unique solution $u\in L^{2}(\Omega)$, and there is a resolvent estimate in the $L^{2}(\Omega)$ norm
\begin{equation}
\label{E:resAOmega}
 \|(\lambda\Id-A_{\Omega})^{-1}\| \le \dist(\lambda,\cC)^{-1}\,.
\end{equation}
\end{proposition}
\begin{remark}
 \label{R:StabGal}
The same argument implies stability for any Galerkin method: Let $X_{h}$ be any closed subspace of $L^{2}(\Omega)$, and let $A_{h}:X_{h}\to X_{h}$ be the operator defined by restricting the sesquilinear form $(u,Av)$ to $X_{h}\times X_{h}$. Then the statement of Proposition~\ref{P:W(AOmega)} remains true if we replace $A_{\Omega}$ by $A_{h}$.
\end{remark}
\begin{remark}
 \label{R:WAOmega=WA}
 Whereas there is, in general, no simple relation between the spectra $\Sp(A_{\Omega})$ and $\Sp(A)$, for the numerical ranges of our convolution operators with homogeneous kernels we not only have the inclusion
 $W(A_{\Omega})\subset W(A)$, but also the converse. Namely there holds
 \begin{equation}
\label{E:WAOmega=WA}
  \overline{W(A_{\Omega})} = \overline{W(A)}\quad\mbox{ for any open subset $\Omega\subset\C$.}
 \end{equation}
\begin{proof}
  The set of Rayleigh quotients $\frac{(u,Au)}{(u,u)}$, where 
  $u\in L^{2}(\R^{d})\setminus\{0\}$ has compact support, is a dense subset of $W(A)$. 
 We show that it is a subset of $W(A_{\Omega})$: Indeed,
  let $u$ be such a function and let $\rho>0$ and $a\in\R^{d}$ be chosen such that the support of the function 
$u_{\rho,a}$ defined by
$u_{\rho,a}(x)=u(\rho x +a)$ is contained in $\Omega$. Then
$$
\frac{(u,Au)}{(u,u)}=\frac{(u_{\rho,a},Au_{\rho,a})}{(u_{\rho,a},u_{\rho,a})}\in W(A_{\Omega}).
$$
\end{proof}
\end{remark}

%%%%%%%%%
\subsubsection{Special kernels}\label{SSS:Specific}
%%%%%%%%%
For $d=1$, there is essentially only one non-trivial kernel homogeneous of degree $-d$, namely $K(x)=\frac1x$.

In $\R^{d}$ for $d\ge2$, while some of the following analysis would be possible for more general homogeneous kernels, we focus now on the situation \eqref{E:Kernels}. 
This means that from now on, we fix a strongly singular kernel $K$ and a homogeneous polynomial $p$ of degree $2$ with $K(x)=p(x)|x|^{-d-2}$, satisfying \eqref{E:Kvanish}, considered as a distribution on $\R^{d}$ according to \eqref{E:pvK}, and we denote by $\hat K$ its Fourier transform. 
\begin{lemma}
 \label{L:Khat} 
 Let $K$ have the form \eqref{E:Kernels} and satisfy \eqref{E:Kvanish}. Then
 \begin{equation}
\label{E:FKhat}
   \hat K(\xi) = -\nu_{d} \frac{p(\xi)}{|\xi|^{2}}, \quad\mbox{ where }
   \nu_{d} = \frac{2\pi^{\frac d2}}{d\,\Gamma(\frac d2)}
   \mbox{ is the volume of the unit ball in }\R^{d}\,.
\end{equation}
\end{lemma}
\begin{proof}
We first compute the Fourier transform of $p(x)e^{-|x|^{2}s}$, using \eqref{E:FGauss}
$$
 \sF_{x\mapsto\xi}[p(x)e^{-|x|^{2}s}] = 
 (\frac\pi s)^{\frac d2} p(-i\partial_{\xi}) e^{-\frac{|\xi|^{2}}{4s}} \,.
$$
The evaluation of these derivatives leads to the following simple result, as we will show:
\begin{equation}
\label{E:Fpexp}
 \sF_{x\mapsto\xi}[p(x)e^{-|x|^{2}s}] = -(\frac\pi s)^{\frac d2}\frac1{4s^{2}} p(\xi) e^{-\frac{|\xi|^{2}}{4s}}\,.
\end{equation}
For $j,k\in\{1,\dots,d\}$ with $j\ne k$, let
\begin{equation}
\label{E:ajk,bjk}
  a_{jk}(x)=x_{j}^{2}-x_{k}^{2}\,,\;\quad  b_{jk}(x)=x_{j}x_{k}\,.
\end{equation}
Any homogeneous polynomial of degree $2$ satisfying the cancellation condition $\int_{\SS^{d-1}}p=0$ is a linear combination of the $a_{jk}$ and $b_{jk}$, so we need to verify \eqref{E:Fpexp} only for these.

Note that $\partial_{\xi_{j}}e^{-\frac{|\xi|^{2}}{4s}} 
 = -\frac1{2s}\xi_{j} e^{-\frac{|\xi|^{2}}{4s}}$ and
$\partial_{\xi_{j}}^{2}e^{-\frac{|\xi|^{2}}{4s}} 
 = \big(-\frac1{2s}+ \frac1{4s^{2}}\xi_{j}^{2}\big) e^{-\frac{|\xi|^{2}}{4s}}$.
 
Then for $p=a_{jk}$, we see
$$
  (\partial_{\xi_{k}}^{2}-\partial_{\xi_{j}}^{2}) e^{-\frac{|\xi|^{2}}{4s}} =
  \frac1{4s^{2}}(\xi_{k}^{2}-\xi_{j}^{2}) e^{-\frac{|\xi|^{2}}{4s}}\,,
$$
and for $p=b_{jk}$, we see
$$
  \partial_{\xi_{j}}\partial_{\xi_{k}} e^{-\frac{|\xi|^{2}}{4s}} =
  \frac1{4s^{2}}\xi_{k}\xi_{j} e^{-\frac{|\xi|^{2}}{4s}}\,.
$$
Thus in both cases we have
\begin{equation}
\label{E:p(dxi)Gauss}
 p(\partial_{\xi}) e^{-\frac{|\xi|^{2}}{4s}} =
 \frac1{4s^{2}} p(\xi) e^{-\frac{|\xi|^{2}}{4s}}\,,
\end{equation}
and \eqref{E:Fpexp} is proved.

Now we use the definition of the Gamma function
$$
 \Gamma(a) = \int_{0}^{\infty}t^{a}e^{-t}\tfrac{dt}{t}
  = |x|^{2a}\int_{0}^{\infty} s^{a}\,e^{-|x|^{2}s}\tfrac{ds}{s}
$$
to write the kernel as an integral over Gaussians
\begin{equation}
\label{E:K=Gauss}
 K(x) = p(x)|x|^{-d-2} = 
 \frac1{\Gamma(\frac d2+1)}\int_{0}^{\infty} s^{\frac d2+1}p(x)e^{-|x|^{2}s}\tfrac{ds}{s}\,.
\end{equation}
Taking Fourier transforms and using \eqref{E:Fpexp}, we find with $u=\frac{|\xi|^{2}
}{4s}$
\begin{equation}
\label{E:KhatGauss}
-\hat K(\xi) = 
\frac{\pi^{\frac d2}}{4\Gamma(\frac d2+1)}\int_{0}^{\infty}\!\!\!\! s^{-1}p(\xi)e^{-\frac{|\xi|^{2}}{4s}}\tfrac{ds}{s}
= p(\xi) |\xi|^{-2} \frac{\pi^{\frac d2}}{\frac d2\Gamma(\frac d2)}\int_{0}^{\infty}\!\!\!\!  e^{-u}du
= \nu_{d}  \,p(\xi) \,|\xi|^{-2}
\end{equation}
as claimed.
\end{proof}

%%%%%%%%%
\subsection{Delta-delta discretization}\label{SS:DDD}\ 
%%%%%%%%%

Let $N\in\N$, fix some origin $a^{N}\in\R^{d}$ and define the cubic grid of meshwidth $h=\frac1N$ by
$$
 \Sigma^{N}=\{x^{N}_{m}=a^{N}+\frac{m}N \mid m\in\Z^{d}\}\,.
$$ 
We further define 
$$
 \omega^{N} = \{m\in\Z^{d} \mid x^{N}_{m}\in \Omega\}\,.
$$
Then a very simple discretization of the strongly singular integral equation 
$(\lambda \Id -A_{\Omega})u=f$
\eqref{E:SIE} is the following
\begin{equation}
\label{E:DDA}
 \lambda u_{m} - N^{-d}\!\!\!\!\sum_{n\in\omega^{N},m\ne n}K(x^{N}_{m}-x^{N}_{n})u_{n}= f(x^{N}_{m}) \,,\quad (m\in\omega^{N})\,,
\end{equation}
or in shorthand $(\lambda\Id - T^{N})U=F$.

The name ``delta-delta discretization'' points at the fact that this discretization formally looks like a Galerkin method for the integral equation \eqref{E:SIE} with Dirac deltas as both test and trial functions, except for the diagonal terms of $T^{N}$, where we put zero, which is natural in view of the cancellation condition \eqref{E:Kvanish}.

Our aim in this paper is to analyze the linear system \eqref{E:DDA}, in particular its stability in $\ell^{2}(\omega^{N})$, in the same way as we did above for the integral equation \eqref{E:SIE} in $L^{2}(\Omega)$, and to compare the two.

We state a general result here, which we prove in the next section. More precise results will be given below in Section~\ref{S:Examples} for some examples, in particular those mentioned in Subsection~\ref{SS:motiv}.
\begin{theorem}
 \label{T:stabgeneral}
Let $K$ be a strongly singular kernel satisfying \eqref{E:Kernels} and \eqref{E:Kvanish}. Then there exists a compact convex set $\cC\subset\C$ such that for any $\lambda\in\C\setminus\cC$, any $N\in\N$ for which $\omega^{N}$ is non-empty, and for any $F\in\ell^{2}(\Omega^{N})$, the system \eqref{E:DDA} has a unique solution, and there is a uniform estimate for the inverse in the $\ell^{2}(\Omega^{N})$ operator norm
\begin{equation}
\label{E:resTN}
 \|(\lambda\Id-T^{N})^{-1}\|_{\cL(\ell^{2}(\Omega^{N}))} \le \dist(\lambda,\cC)^{-1}\,.
\end{equation}
Furthermore, with the strongly singular integral operator $A$ defined above in Section~\ref{SSS:Homog}, there holds the inclusion
\begin{equation}
\label{E:WAsubsetC}
  W(A) \subset \cC\,.
\end{equation}
\end{theorem}
\begin{remark}
 \label{R:overshoot}
Note that the inclusion $W(A) \subset \cC$ implies that for $\lambda\not\in\cC$ the singular integral equation is uniquely solvable, too, and provides the a priori estimate \eqref{E:resAOmega}, for any domain $\Omega\subset\R^{d}$. On the other hand, in order to guarantee stability for $\lambda\in\C\setminus\cC$, the inclusion may need to be strict, as we shall see in the examples, and then there may be $\lambda\in\cC\setminus W(A)$ for which the singular integral equation is well posed, but the delta-delta discretization is \emph{unstable}.
\end{remark}

%%%%%%%%%
\section{The discrete system}\label{S:discrete}
%%%%%%%%%
Let $T^{N}$ be the matrix representing the discretized integral operator in \eqref{E:DDA}:
\begin{equation}
\label{E:TN}
T^{N} = (t^{N}_{mn})_{m,n\in\omega^{N}}
\quad\mbox{ with }\quad
 t^{N}_{mn} = \begin{cases}
                N^{-d}\, K(x^{N}_{m}-x^{N}_{n}) &(m\ne n)\\
                0 &(m=n)
 				\end{cases}
\,.
\end{equation}
Our aim is to bound the numerical range $W(T^{N})$ independently of $N$.

%%%%%%%%%
\subsection{Toeplitz structure}\label{SS:Toeplitz}\ 
%%%%%%%%%

The matrix elements $t^{N}_{mn}$ of $T^{N}$ do not depend on the choice of the origin $a^{N}$, and
since we assumed that $K$ is homogeneous of degree $-d$, we have
$$
  N^{-d}\, K(x^{N}_{m}-x^{N}_{n}) = K(m-n)\,,
$$
hence $T^{N}$ is a finite section of a fixed infinite Toeplitz (discrete convolution) matrix 
\begin{equation}
\label{E:Tinf}
 T = (t_{mn})_{m,n\in\Z^{d}}
\quad\mbox{ with }\quad
 t_{mn} = \begin{cases}
                K(m-n) &(m\ne n)\,,\\
                0 &(m=n)\,.
 				\end{cases}
\end{equation}
Theorem~\ref{T:stabgeneral} will be proved if we can show that $T$ defines a bounded linear operator in $\ell^{2}(\Z^{d})$ whose numerical range $W(T)$ contains $W(A)$. We can then choose $\cC$ as the closure of $W(T)$.

We use Fourier series and the convolution theorem to diagonalize the matrix $T$ and to  represent the sesquilinear form defined by the matrix $T^{N}$, compare \eqref{E:Par+conv,series}. 
For $U=(u_{m})_{m\in\omega^{N}}$, we find
\begin{equation}
\label{E:formTN}
 (U,T^{N}U) = (2\pi)^{-d}\int_{Q} F(\tau)\,|\tilde u(\tau)|^{2}\,d\tau\,.
\end{equation}
Here $\tilde u(\tau)=\sum_{m\in\omega^{N}}u_{m}e^{im\cdot \tau}$ and $Q=[-\pi,\pi]^{d}$. $F(\tau)$ is the symbol (characteristic function) of the Toeplitz matrix $T$:
\begin{equation}
\label{E:F(t)}
 F(\tau) = \sum_{m\in\Z^{d},m\ne0}K(m)\,e^{im\cdot \tau}\,.
\end{equation}
The problem is now reduced to the study of the operator of multiplication by the function $F$ in $L^{2}(Q)$.
\begin{lemma}
 \label{L:TbyF}
 The operator $T:\ell^{2}(\Z^{d}) \to \ell^{2}(\Z^{d})$ is bounded if and only if $F\in L^{\infty}(Q)$.\\
 The closure of $W(T)$ is the closed convex hull of the range $\im(F)=\{F(\tau)\mid \tau\in Q\}$ and is also equal to the closure of the union 
 $\bigcup_{N\in\N}W(T^{N})$\,. 
\end{lemma}
Proof: This is immediate from \eqref{E:formTN}.

%%%%%%%%%
\subsection{Ewald method}\label{SS:Ewald}\ 
%%%%%%%%%

The problem that makes the statement $F\in L^{\infty}(Q)$ non trivial is that the Fourier series \eqref{E:F(t)} is not absolutely convergent. The sequence $(K(m))_{m\in\Z^{d}}$ is of order $O(|m|^{-d})$ at infinity and therefore in $\ell^{p}(\Z^{d})$ for all $p>1$, but not for $p=1$. Its membership in $\ell^{2}(\Z^{d})$ implies, for example, that the series converges in the sense of $L^{2}(Q)$. 
The slow convergence of the Fourier series for $F$ makes it also unsuitable for using it in numerical computations to find bounds for $\im(F)$. 

We will use a variant of a method introduced by P.~P.~Ewald~\cite{Ewald1921} in 1921 as a tool to compute slowly converging lattice sums. It has become a routine method for the computation of periodic and quasi-periodic Green functions, with application in numerical electrodynamics and other fields where periodic structures appear. Among the many presentations of the method: Appendix A of the article \cite{Essmann-et-al_1995} or Section 2.13.3 in the book \cite{Ammari-et-al_2018}.

We use it here as a summation method for our slowly converging Fourier series. In our restricted setting it turns out to give surprisingly simple results.

The method introduces a decomposition $K=K^{F}+K^{P}$ for the coefficients and correspondingly $F=F^{F}+F^{P}$ for the Fourier series in such a way that both $K^{F}$ and the Fourier transform $\hat K^{P}$ of $K^{P}$ are exponentially decreasing at infinity, so that both the Fourier series for $F^{F}(\tau)$ and the Poisson sum (compare \eqref{E:Pois}) for $F^{P}(\tau)$ are rapidly convergent, which not only proves the boundedness of $F$, but gives also a fast numerical algorithm for its computation. 

In the literature one often labels the two terms in the decomposition ``spatial'' and ``spectral'' sums, but this is not pertinent to our situation, where the lattice sum runs over the Fourier variable, and the Fourier series runs overs spatial points. So we will use ``Fourier'' and ``Poisson'' sums as labels.

The idea of Ewald's method is to represent $K(x)$ by an integral over Gaussians from $0$ to $\infty$ as we did already in Section~\ref{SSS:Specific} above:
\begin{equation}
\label{E:intK}
 K(x) = p(x)|x|^{-d-2} = 
 \frac{p(x)}{\Gamma(\frac d2+1)}\int_{0}^{\infty} s^{\frac d2}e^{-|x|^{2}s}\,ds
\end{equation}
and then to split the integral at a point $\beta^{2}>0$:
\begin{align}
\label{E:intKF}
K^{F}(x) &= \frac{p(x)}{\Gamma(\frac d2+1)}\int_{\beta^{2}}^{\infty} s^{\frac d2}e^{-|x|^{2}s}\,ds\,,\\
\label{E:intKP}
K^{P}(x) &= \frac{p(x)}{\Gamma(\frac d2+1)}\int_{0}^{\beta^{2}} s^{\frac d2}e^{-|x|^{2}s}\,ds\,.
\end{align}
We see that $K^{F}$ is simply the product of $K$ by a function exponentially decreasing at infinity
\begin{equation}
\label{E:KF=KxGamma}
  K^{F}(x) = K(x) \, \frac{\Gamma(\tfrac d2+1,\beta^{2}|x|^{2})}{\Gamma(\frac d2+1)}
\end{equation}
with the (upper) incomplete Gamma function (see \cite[\S 6.5]{Abramowitz-Stegun_1964})
$$
  \Gamma(a,x)=\int_{x}^{\infty}t^{a-1}e^{-t}\,dt\,.
$$
Therefore 
 $K^{F}(x) = O(|x|^{2}e^{-\beta^{2} |x|^{2}})$ as $|x|\to\infty$\,, and the Fourier series for $F^{F}(\tau)$
\begin{equation}
\label{E:serFF}
 F^{F}(\tau) =  \sum_{m\in\Z^{d},m\ne0}K^{F}(m)\,e^{im\cdot \tau}
\end{equation}
converges rapidly, implying that $F^{F}$ is an analytic function on $\R^{d}/(2\pi\Z)^{d}$.

Consequently, the Fourier series for $F^{P}(\tau)$ converges as slowly as the one for $F(\tau)$, and we use instead the Poisson summation formula \eqref{E:Pois} and write
\begin{equation}
\label{E:serFP}
 F^{P}(\tau) = \sum_{n\in\Z^{d}}\hat K^{P}(\tau+2\pi n)\,.
\end{equation}
We can evaluate $\hat K^{P}$ with the formulas used for $\hat K$ in Lemma~\ref{L:Khat}. As in \eqref{E:KhatGauss} we obtain
\begin{equation}
\begin{aligned}
\label{E:KPhat=KhatxG}
 \hat K^{P}(\xi) &=
 \frac{-\pi^{\frac d2}}{4\Gamma(\frac d2+1)}\int_{0}^{\beta^{2}}\!\!\!\! s^{-1}p(\xi)e^{-\frac{|\xi|^{2}}{4s}}\tfrac{ds}{s}
=   -p(\xi) |\xi|^{-2} \frac{\pi^{\frac d2}}{\frac d2\Gamma(\frac d2)}\int_{\frac{|\xi|^{2}}{4\beta^{2}}}^{\infty} e^{-u}du\\
&= \hat K(\xi)  e^{-\frac{|\xi|^{2}}{4\beta^{2}}}\,.
\end{aligned}
\end{equation}
Therefore we also obtain a very simple form for the Fourier transform, namely that $\hat K^{P}$ is just the symbol of $A$ cut off at infinity, and therefore the series \eqref{E:serFP} converges absolutely and uniformly. At most one term in the sum may be discontinuous, when $\tau+2\pi n=0$, and for $\tau\in Q$ this is the term with $n=0$.
We can summarize the result.
\begin{proposition}
 \label{P:NumSym}
The symbol $F(\tau)$ of the infinite Toeplitz matrix $T$ is a bounded function given for any $\beta>0$ by the exponentially convergent sums
\begin{equation}
\label{E:Fewald}
 F(\tau) = 
 \sum_{m\in\Z^{d},m\ne0}K(m)\tfrac{\Gamma(\tfrac d2+1,\beta^{2}|m|^{2})}{\Gamma(\frac d2+1)}e^{im\cdot \tau} +
 \sum_{n\in\Z^{d}}\hat K(\tau+2\pi n)\,e^{-\frac{|\tau+2\pi n|^{2}}{4\beta^{2}}}\,.
\end{equation}
In the period cube $Q=[-\pi,\pi]^{d}$, it is $C^{\infty}$ outside of $\,0$, and it has the form
\begin{equation}
\label{E:F=F0+Khat} 
 F(\tau) = \hat K(\tau) + F_{0}(\tau) \quad\mbox{ where $F_{0}$ is analytic in $Q$ and $F_{0}(0)=0$}.  
\end{equation}
\end{proposition}
\begin{proof}
We have proved equation \eqref{E:Fewald} above, except for one point: From Poisson's summation formula follows that the Poisson sum \eqref{E:serFP} equals the Fourier series with coefficients $K^{P}(m)$, $m\in\Z^{d}$, including $m=0$. But in the Fourier series \eqref{E:F(t)} defining $F(t)$ as well as in \eqref{E:serFF} defining $F^{F}(t)$, we have excluded $m=0$. So we should compensate for $K^{P}(0)$, which according to \eqref{E:intKP} equals
$$
 K^{P}(0) = \frac{p(0)\beta^{d}}{\Gamma(\frac d2+1)}\,.
$$
Now, since we assumed $p(x)$ to be a homogeneous polynomial of degree $2$, we have $p(0)=0$ and hence no compensation is needed.

Representing $F_{0}$ as
$$
 F_{0}(\tau)=
  \sum_{m\in\Z^{d},m\ne0}K(m)\tfrac{\Gamma(\tfrac d2+1,\beta^{2}|m|^{2})}{\Gamma(\frac d2+1)}e^{im\cdot \tau} +
 \sum_{n\in\Z^{d},n\ne0}\hat K(\tau+2\pi n)e^{-\frac{|\tau+2\pi n|^{2}}{4\beta^{2}}}
 + \hat K(\tau)\big(e^{-\frac{|\tau|^{2}}{4\beta^{2}}}-1\big)\,,
$$
we see immediately that it is analytic. For finding $F_{0}(0)$, we can use the following observation.
\begin{lemma}
 \label{L:cubicsym}
Let $S\subset\R^{d}$ be a finite set that is cubically symmetric, i.\ e.\ invariant under reflections at coordinate planes and under permutations of the coordinates, and let $p$ be a homogeneous polynomial of degree $2$ satisfying the cancellation condition $\int_{\SS^{d-1}}p=0$. Then
$$
 \sum_{x\in S}p(x) = 0.
$$
\end{lemma}
This is immediately clear when $p$ is one of the $a_{jk}$ or $b_{jk}$ from \eqref{E:ajk,bjk}, and it is therefore true for all $p$ satisfying the (spherical) cancellation condition.

For any $M\in\R$, the set $\{m\in\Z^{d}\mid |m|^{2}=M\}$ is either empty or cubically symmetric.
Therefore for $\tau=0$, the two sums in the representation of $F_{0}(\tau)$ are $0$. The last term
$$
 \hat K(\tau)\big(e^{-\frac{|\tau|^{2}}{4\beta^{2}}}-1\big) 
 = -\nu_{d}\,p(\tau)\,\frac{e^{-\frac{|\tau|^{2}}{4\beta^{2}}}-1}{|\tau|^{2}}
$$ 
tends to $\nu_{D}p(0)/(4\beta^{2})=0$ as $\tau\to0$, and hence $F_{0}(0)=0$.
\end{proof}
Poposition~\ref{P:NumSym} implies Theorem~\ref{T:stabgeneral}, where $\cC$ is the closed convex hull of $\im(F)$. The inclusion $W(A)\subset\cC$ is easy to see from \eqref{E:F=F0+Khat}:\\
Given $\epsilon>0$, let $\delta>0$ be such that for $|\tau|<\delta$ we have $|F_{0}(\tau)|<\epsilon$.
Since $\hat K$ is homogeneous of degree zero, it takes all of its values already on the ball $B_{\delta}(0)$ of radius $\delta$. Thus
$$
 \im(\hat K) \subset F(B_{\delta}(0)) + B_{\epsilon}(0) \subset  \im(F) + B_{\epsilon}(0)\,.
$$
Taking convex hulls shows that
$$
 W(A) \subset \cC + B_{\epsilon}(0) \quad\mbox{ for all }\;\epsilon>0\,.
$$

\begin{remark}
 \label{R:simple}
The very simple form of the Ewald representation \eqref{E:Fewald} comes from the very simple form of the Fourier transforms \eqref{E:FKhat} and \eqref{E:Fpexp}, which in turn rely on the cancellation condition \eqref{E:Kvanish}. Now for the kernel $K$ this condition is natural, because it is necessary in order to represent $K$ as a homogeneous distribution and to have a bounded Fourier transform. But for the symbol $\hat K$ it is not as natural. We can add a constant and still have a function homogeneous of degree zero, which will then not satisfy the cancellation condition. An example is $\xi_{j}\xi_{k}|\xi|^{-2}$ for all $j,k$, even for $j=k$.

On the other hand, the representation $K(x)=p(x)|x|^{-d-2}$ may not be the most natural, one may come across cases (see Example~\ref{Ex:5} below) like 
$$K_{jk}(x)=\delta_{jk}|x|^{-d}-d\,x_{j}x_{k}|x|^{-d-2}\,,
$$ 
where for $j=k$ the two terms in the sum do not separately satisfy \eqref{E:Kvanish}. This fits into our framework, however, because
$$
K_{jk}(x)=-d\,b_{jk}(x)|x|^{-d-2}\quad\mbox{ for $j\ne k$, \quad and }\quad
K_{kk}(x)=\sum_{j=1}^{d}a_{jk}(x)|x|^{-d-2}\,.
$$  
If one treats the two terms individually, one may get formulas for Fourier transforms and for the Ewald splitting that are less symmetric than what we presented above.
\end{remark}

%%%%%%%%%%
\subsection{An integral representation}\label{SS:FInt}\ \nopagebreak
%%%%%%%%%

We have another look at the Ewald splitting for the numerical symbol $F(\xi)=F^{F}(\xi)+F^{P}(\xi)$ described in \eqref{E:intKF}--\eqref{E:Fewald}
\begin{align}
\label{E:FFsumint}
 F^{F}(\xi) &= \sum_{m\in\Z^{d}} 
    \frac{p(m)}{\Gamma(\frac d2+1)}\int_{\beta^{2}}^{\infty} s^{\frac d2}e^{-|m|^{2}s}\,ds
    \,e^{im\cdot \xi} \\
\label{E:FPsumint}
 F^{P}(\xi) &= \sum_{n\in\Z^{d}}
    \frac{-\pi^{\frac d2}}{4\Gamma(\frac d2+1)}\int_{0}^{\beta^{2}}\!\!\!\! s^{-2}p(\xi+2\pi n)e^{-\frac{|\xi+2\pi n|^{2}}{4s}}\,ds\,.
\end{align}
These formulas are valid for any $0<\beta<\infty$. All the sums and integrals are converging absolutely here, and therefore we can interchange sums and integrals.
\begin{align}
\label{E:FFintsum}
 F^{F}(\xi) &= \int_{\beta^{2}}^{\infty}\!\!\! H^{F}(\xi,s)\,ds \quad\mbox{ with }\quad
 H^{F}(\xi,s) = \sum_{m\in\Z^{d}} 
    \frac{p(m)}{\Gamma(\frac d2+1)} s^{\frac d2}e^{-|m|^{2}s}
    \,e^{im\cdot \xi}\\
\label{E:FPintsum}
 F^{P}(\xi) &= \int_{0}^{\beta^{2}}\!\!\! H^{P}(\xi,s)\,ds \quad\mbox{ with }\quad
 H^{P}(\xi,s) =
 \sum_{n\in\Z^{d}}
    \frac{-\pi^{\frac d2}p(\xi+2\pi n)}{4\Gamma(\frac d2+1)} s^{-2}e^{-\frac{|\xi+2\pi n|^{2}}{4s}}\,.
\end{align}
From the definition \eqref{E:FFintsum} of $H^{F}$ and the fact that $|m|\ge1$ in the sum follows without difficulty that for any $0<\gamma<1$ there exists a constant $C$ such that
\begin{equation}
\label{E:HFstoinf} |H^{F}(\xi,s)| \le C\,e^{-\gamma s} \quad\mbox{ for all }s\ge1,\; \xi\in\R^{d}.
\end{equation}
To see the behavior of $H^{P}(\xi,s)$  from \eqref{E:FPintsum}, we decompose
$$
 H^{P}(\xi,s) = H_{0}(\xi,s) + H_{1}(\xi,s)
$$
with 
\begin{align}
\label{E:H0} 
 H_{0}(\xi,s) &= -\tfrac{\pi^{\frac d2}}{4\Gamma(\frac d2+1)}
 \!\!\!
  \sum_{n\in\Z^{d},n\ne0} \!\!\!
    p(\xi+2\pi n) s^{-2}e^{-\frac{|\xi+2\pi n|^{2}}{4s}}\,,\\
\label{E:H1}
  H_{1}(\xi,s) &= -\tfrac{\pi^{\frac d2}}{4\Gamma(\frac d2+1)}
     p(\xi) s^{-2}e^{-\frac{|\xi|^{2}}{4s}}\,.
\end{align}
Now we use the fact that for $\xi\in Q$ and $n\ne0$ we have $|\xi+2\pi n|\ge\pi$. 
Therefore for any $\delta<\frac{\pi^{2}}{4}$ there is a constant $C$ such that
\begin{equation}
\label{E:H0at0}
 |H_{0}(\xi,s)| \le C\, e^{-\frac\delta s} \quad\mbox{ for all }0<s\le1,\; \xi\in Q,
\end{equation}
and $H_{0}(\xi,s)$ is analytic in $\xi$ for all $s$.

It remains to analyze the term with $n=0$, i.e. $H_{1}$. It is clear that it vanishes for $\xi=0$, and for every $\xi\ne0$ there exists a constant $C_{\xi}$ and $0<\gamma<\frac{|\xi|^{2}}4$ such that 
\begin{equation}
\label{E:H1alls}
|H_{1}(\xi,s)| \le C_{\xi} \min\{s^{-2}, e^{-\frac\gamma s} \}
\quad\mbox{ for all } s\in(0,\infty)\,.
\end{equation}
Thus $H_{1}(\xi,s)$ is integrable over $s\in(0,\infty)$ for all $\xi$, but there is no uniform bound for $C_{\xi}$: 
Considering $\sup_{s>0}s^{-2}e^{-\frac{|\xi|^{2}}{4s}}$, one sees that 
$C_{\xi}=O(|\xi|^{-2})$ as $\xi\to0$.

Thus we see that $H^{F}$ is integrable as $s\to\infty$ according to \eqref{E:HFstoinf}, and $H^{P}$ is integrable as $s\to0$ according to \eqref{E:H0at0} and \eqref{E:H1alls}, but, because of Poisson's summation formula, they are in fact the same
$$
  H^{F}(\xi,s) = H^{P}(\xi,s)\,,
$$
so we can use all of the above estimates for both of them.
We can summarize
\begin{proposition}
 \label{P:FInt}
 The symbol $F(\xi)$ has the integral representation
\begin{equation}
\label{E:F=intH}
 F(\xi) = \int_{0}^{\infty} H(\xi,s)\,ds\,,
\end{equation}
where $H(\xi,s)$ is given either by the Fourier series $H^{F}$ in \eqref{E:FFintsum} or, equivalently, by the lattice sum $H^{P}$ in \eqref{E:FPintsum}.  
The decomposition $F=F_{0}+\hat K$ in Proposition~\ref{P:NumSym} corresponds to the decomposition $H=H_{0}+H_{1}$ with $H_{0}$ and $H_{1}$ defined in \eqref{E:H0} and \eqref{E:H1}, and there holds
\begin{equation}
\label{E:F01Int}
F_{0}(\xi) = \int_{0}^{\infty} H_{0}(\xi,s)\,ds \qquad
\mbox{ and }\quad
\hat K(\xi) = \int_{0}^{\infty} H_{1}(\xi,s)\,ds \,.
\end{equation}
In these integrals, the functions $s\mapsto H_{0}(\xi,s)$, $s\mapsto H_{1}(\xi,s)$, and $s\mapsto H(\xi,s)$ are integrable on $(0,\infty)$ for any $\xi\in Q$, for any $\xi\in\R^{d}\setminus\{0\}$, and for any $\xi\in Q\setminus\{0\}$, respectively.
\end{proposition}

The integral representations \eqref{E:F=intH} and  \eqref{E:F01Int} will be used below to get bounds for the function $F(\xi)$ from estimates for $H(\xi,s)$. The latter will be a consequence of the following observation that can be proved using  Fourier representations \eqref{E:FFintsum} for $H$ and \eqref{E:Fpexp} for $H_{1}$. 

\begin{lemma}
The functions
 \label{L:heat}
$$
 (\xi,s)\mapsto s^{-\frac d2}H_{0}(\xi,s),\quad
 (\xi,s)\mapsto s^{-\frac d2}H_{1}(\xi,s),\quad
 (\xi,s)\mapsto s^{-\frac d2}H(\xi,s),\quad
$$
are solutions of the heat equation 
$$
(\partial_{s}-\Delta_{\xi})u(\xi,s)=0\quad\mbox{ in }\; Q\times(0,\infty).
$$
\end{lemma}

%%%%%%%%%
\subsection{Matrix-valued kernels}\label{SS:matrix}\ 
%%%%%%%%%

Until now, we have considered kernel functions with values in $\C$ and integral operators acting on scalar functions. The generalization to vector-valued functions and matrix-valued kernels is simple and straightforward, and we do not find it necessary to introduce typographic distinctions for the vector-valued objects. The main difference is that in the general theory of Section~\ref{SS:kernels}, one has to use the numerical range $W(K(x))$ of the matrix $K(x)$ instead of the value $K(x)$ in statements such as Lemma~\ref{L:Sp+W(A)} and Proposition~\ref{P:W(AOmega)}. In particular
\begin{equation}
\label{E:NumranGen}
\mbox{$\overline{W(A)}$\; is the closed convex hull of }\;
\bigcup\nolimits_{\xi \in\R^{d}}W(\hat K(\xi))\,.
\end{equation}
Theorem~\ref{T:stabgeneral} remains literally true, but for the construction of the set $\cC$ one has once again to use the numerical range $W(F(t))$ of the matrix-valued function $F$. 
In Lemma~\ref{L:TbyF}, the characterization of the numerical range $W(T)$ is to be understood as follows.
\begin{lemma}
\label{L:TbyFEx5}
 The closure of $W(T)$ is the closed convex hull of $\bigcup_{\tau\in Q}W(F(\tau))$ and is also equal to the closure of the union 
 $\bigcup_{N\in\N}W(T^{N})$.
\end{lemma}
 
The basic Parseval-convolution formula \eqref{E:Par+conv,series} now has to be written, instead of the scalar version \eqref{E:formTN}, as
\begin{equation}
\label{E:formTNEx5}
 (U,T^{N}U) = (U,TU) = (2\pi)^{-d}\int_{Q} \overline{\tilde u(\tau)}^{\top} F(\tau)\,\tilde u(\tau)\,d\tau\,.
\end{equation}
Here $\tilde u(\tau)=\sum_{m\in\omega^{N}}u_{m}e^{im\cdot \tau}$, and $F(\tau)$ is the matrix-valued symbol of the block Toeplitz matrix 
$T=\big(K(m-n)\big)_{m,n\in\Z^{d}}$:
\begin{equation}
\label{E:F(t)Ex5}
 F(\tau) = \sum_{m\in\Z^{d},m\ne0}K(m)\,e^{im\cdot \tau}\,.
\end{equation}
From \eqref{E:formTNEx5} one can immediately read the properties of the numerical range stated in Lemma~\ref{L:TbyFEx5}.

 In this paper, most considered examples of kernels are real-valued and the matrices symmetric, in which case the integral operators are selfadjoint, and the numerical ranges consist of intervals in the real line.

%%%%%%%%%
\section{Examples}\label{S:Examples}
%%%%%%%%%
\newcounter{Ex}

\refstepcounter{Ex}\label{Ex:1}
%%%%%%%%%
\subsection{Example 1. Dimension $d=1$. Finite Hilbert transformation}\label{SS:1d}\ 
%%%%%%%%%

We start with the simplest example of a strongly singular integral equation and show that the stability of its delta-delta approximation can be completely analyzed, resulting in a kind of ideal stability theorem.

%%%%%%%%%%
\subsubsection{The singular integral equation}\label{SS:SIE1D}
%%%%%%%%%
Let $a,b\in\R$ with $a<b$. On the interval $\Omega=(a,b)$ we consider the singular integral equation, abbreviated as  $(\lambda\Id-A_{\Omega})u=f$,
\begin{equation}
\label{E:IE1D}
 \lambda u(x) - \frac1{i\pi}\int_{\Omega}\frac{u(y)}{x-y}dy = f(x)\,,\quad x\in\Omega\,.
\end{equation}
The integral is understood in the Cauchy principal value sense. 
The kernel function $K(x)=\frac1{i\pi x}$ has the Fourier transform
$$
 \hat K(\xi)=\sign\,\xi.
$$
The operator $A$ of convolution with $K$ on $\R$ is the Hilbert transformation. It satisfies $A^{2}=\Id$, and its spectrum (in a large class of function spaces, for instance $L^{p}(\R)$ with $1<p<\infty$) is $\{-1,1\}$, consisting of two eigenvalues of infinite multiplicity. 

The finite Hilbert transformation $A_{\Omega}$ and its spectral theory are also well studied classical objects, see for example \cite{KoppelmanPincus1959}. Here the spectrum depends on the function space; for $L^{p}(\Omega)$ it is strongly dependent on $p$, but not on $\Omega$, as long as $\Omega$ is a proper subinterval of $\R$. For $p=2$ one has the following description.
\begin{lemma}
\label{L:SpAOmega}
The finite Hilbert transformation $A_{\Omega}$ is a bounded selfadjoint operator in $L^{2}(\Omega)$, unitarily equivalent to the operator of multiplication by $\sigma$ in $L^{2}(-1,1)$ with $\sigma(\xi)=\xi$.
Both the spectrum $\Sp(A_{\Omega})$ and the closure of the numerical range $\overline{W(A_{\Omega})}$ are equal to 
 $\cC=[-1,1]$.
 For all $\lambda\in\C\setminus\cC$ and any $f\in L^{2}(\Omega)$, the integral equation \eqref{E:IE1D} has a unique solution $u\in L^{2}(\Omega)$, and for the resolvent one has in the $L^{2}(\Omega)$ operator norm
\begin{equation}
\label{E:resAOmega1D}
 \|(\lambda\Id-A_{\Omega})^{-1}\| = \dist(\lambda,\cC)^{-1}\,.
\end{equation}
\end{lemma} 
Explicit formulas for the resolvent are known. For the infinite Hilbert transformation this is trivially obtained by algebra:
$$
 (\lambda\Id - A)^{-1} = 
 \frac1{\lambda^{2}-1} (\lambda\Id +  A)\,,
$$
and for the finite Hilbert transformation, formulas for the resolvent can be found for example in \cite{Soehngen1954} or \cite{Tricomi1957}.

%%%%%%%%%%
\subsubsection{The discrete system}\label{SS:DA1D}
%%%%%%%%%
We use the notation of Section \ref{SS:DDA} with $d=1$, in particular $x^{N}_{m}=a^{N}+\frac{m}N$ and $\omega^{N} = \{m\in\Z \mid x^{N}_{m}\in \Omega\}$. The simple delta-delta discretization of our singular integral equation \eqref{E:IE1D} is
\begin{equation}
\label{E:DDA1D}
 \lambda u_{m} - \frac1{i\pi N}\sum_{n\in\omega^{N},m\ne n}\frac{u_{n}}{x^{N}_{m}-x^{N}_{n}}= f(x^{N}_{m}) \,,\quad (m\in\omega^{N})\,.
\end{equation}
The system matrix $T^{N}$ with matrix elements $\frac1{i\pi N}\frac1{x^{N}_{m}-x^{N}_{n}}$ ($m,n\in\omega^{N}$) is a finite section of the infinite Toeplitz matrix
$$
  T = \big( \frac1{i\pi(m-n)} \big)_{m,n\in\Z}\quad
  \mbox{ with zero on the diagonal.}
$$
The symbol $F(\tau)$ is now given by the Fourier series
\begin{equation}
\label{E:F1D}
 F(\tau) = \sum_{m\in\Z,m\ne0}\frac{e^{im\tau}}{i\pi m}
      =  \sum_{m=1}^{\infty} \frac{2\sin m\tau}{\pi m}\,,\quad \tau\in Q=[-\pi,\pi]\,.
\end{equation}
This series converges for all $t\in Q$ to the well known saw-tooth function
\begin{equation}
\label{E:sawtooth}
 F(\tau) = \sign \,\tau - \frac \tau\pi\quad (\tau\ne0)\,,\quad F(0)=0\,.
\end{equation}
The range of this function is the interval $(-1,1)$.

Properties of the matrix $T$ follow immediately from this symbol $F$ and can be summarized as follows.
\begin{lemma}
\label{L:SpT}
The infinite Toeplitz matrix $T$ defines a bounded selfadjoint operator in $\ell^{2}(\Z)$, unitarily equivalent to the operator of multiplication by $F$ in $L^{2}(-\pi,\pi)$ with $F$ given in \eqref{E:sawtooth}.
Both the spectrum $\Sp(T)$ and the closure of the numerical range $\overline{W(T)}$ are equal to 
 $\cC=[-1,1]$.
 For all $\lambda\in\C\setminus\cC$ the operator $\lambda\Id-T$ is invertible in $\ell^{2}(\Z)$, and for the resolvent one has in the $\ell^{2}(\Z)$ operator norm
\begin{equation}
\label{E:resT1D}
 \|(\lambda\Id-T)^{-1}\| = \dist(\lambda,\cC)^{-1}\,.
\end{equation}
\end{lemma} 
\begin{corollary}
 \label{C:StabTN1D}
 The matrix $T^{N}$ of the system \eqref{E:DDA1D} is selfadjoint with its eigenvalues in $\cC=[-1,1]$.  For $\lambda\in\C\setminus\cC$, there is a uniform resolvent estimate in the $\ell^{2}$ operator norm
\begin{equation}
 \label{E:resTN1D}
 \|(\lambda\Id-T^{N})^{-1}\| \le \dist(\lambda,\cC)^{-1}\,.
\end{equation}
\end{corollary}
The converse is also true: If there is a uniform stability estimate
$$
 \|(\lambda\Id-T^{N})^{-1}\| \le C \quad\mbox{ for all $N$ },
$$
then one also has (by a standard Galerkin argument)
$\|(\lambda\Id-T)^{-1}\|\le C$, hence $\dist(\lambda,\cC)\ge\frac1C$ and $\lambda\not\in\cC$.
Combining this with Lemma~\ref{L:SpAOmega}, we obtain the following description of the stability result for our delta-delta discretisation of the finite Hilbert transform.
\begin{theorem}
 \label{T:Stab1D}
 For $\lambda\in\C$ the following are equivalent:\\
 (i) The singular integral equation \eqref{E:IE1D} has a unique solution $u\in L^{2}(\Omega)$ for any $f\in L^{2}(\Omega)$.\\
 (ii) The discretization method \eqref{E:DDA1D} is stable in the $\ell^{2}$ norm.\\
 (iii) $\lambda\not\in\cC$, where $\cC=[-1,1]$.\\
 For such $\lambda$, there is an estimate for the operator norms 
 \begin{equation}
\label{E:Res1D}
 \|(\lambda\Id-T^{N})^{-1}\|_{\cL(\ell^{2}(\omega^{N}))} \le \|(\lambda\Id-A_{\Omega})^{-1}\|_{\cL(L^{2}(\Omega))}\,.
\end{equation}
\end{theorem}

\refstepcounter{Ex}\label{Ex:2}
%%%%%%%%%
\subsection{Example 2. Dimension $d=2$, kernel $x_{1}x_{2}|x|^{-4}$}\label{SS:2db12}\ 
%%%%%%%%%

We consider now the simplest higher-dimensional example where in the notation of Section~\ref{SSS:Specific} $d=2$ and $p(x)=-\frac1\pi b_{12}(x)$, see \eqref{E:ajk,bjk}. 
We show that the stability of its delta-delta approximation follows a similar simple pattern as in the previous one-dimensional example, although the proof is non-trivial. 

%%%%%%%%%
\subsubsection{The singular integral equation}
%%%%%%%%%
The kernel and its Fourier transform are
\begin{equation}
\label{E:KandKhatforb}
 K(x) = -\frac{x_{1}x_{2}}{\pi|x|^{4}}\,,\qquad
 \hat K(\xi) = \frac{\xi_{1}\xi_{2}}{|\xi|^{2}}\,.
\end{equation}
For $\Omega\subset\R^{2}$, we consider the singular integral equation 
$(\lambda \Id -A_{\Omega})u=f$ as in \eqref{E:SIE}
\begin{equation}
\label{E:IEEx2}
\lambda u(x) - \pv\!\! \int_{\Omega}K(x-y) u(y)\,dy = f(x).
\end{equation}
Observing that the range of the function $\hat K$ is the interval 
$[-\frac12,\frac12]$, we can formulate the result of Proposition~\ref{P:W(AOmega)} as follows
\begin{lemma}
 \label{L:AOmegaEx2}
Let $\cC=[-\frac12,\frac12]$. For $\Omega=\R^{2}$, both the spectrum $\Sp(A_{\Omega})$ and the closure of the numerical range $\overline{W(A_{\Omega})}$ in $L^{2}(\Omega)$ are equal to $\cC$. 
For any open subset $\Omega\subset\R^{2}$, the closure of the numerical range 
in $L^{2}(\Omega)$ satisfies $\overline{W(A_{\Omega})}\subset\cC$, and there is a resolvent estimate in the $L^{2}(\Omega)$ operator norm
\begin{equation}
\label{E:resAOmega2DEx2}
 \|(\lambda\Id-A_{\Omega})^{-1}\| \le \dist(\lambda,\cC)^{-1}\,.
\end{equation}
\end{lemma}

%%%%%%%%%%
\subsubsection{The discrete system}\label{SS:DA2dEx2}
%%%%%%%%%
Let  now $\Omega$ be a bounded domain in $\R^{2}$.
In the notation of Section \ref{SS:DDD} with $d=2$, the regular grid consists of the points $x^{N}_{m}=a^{N}+\frac{m}N$, indexed by $\omega^{N} = \{m\in\Z^{2} \mid x^{N}_{m}\in \Omega\}$. The simple delta-delta discretization of our singular integral equation \eqref{E:IEEx2} is
\begin{equation}
\label{E:DDA2dEx2}
 \lambda u_{m} + \frac{1}{\pi} N^{-2}
 \sum_{n\in\omega^{N},m\ne n} 
 \frac{(x^{N}_{m,1}-x^{N}_{n,1})(x^{N}_{m,2}-x^{N}_{n,2})}{|x^{N}_{m}-x^{N}_{n}|^{4}}
  u_{n}
 = f(x^{N}_{m}) \,,\quad (m\in\omega^{N})\,.
\end{equation}
The system matrix $T^{N}$  is now a finite section of the infinite Toeplitz matrix
$$
  T = -\frac1\pi \Big(  
  \frac{(m_{1}-n_{1})(m_{2}-n_{2})}{|m-n|^{4}}
  \Big)_{m,n\in\Z^{2}}\quad
  \mbox{ with zero on the diagonal.}
$$
Its symbol is therefore given by the Fourier series for $\tau\in Q=[-\pi,\pi]^{2}$
\begin{equation}
\label{E:FEx2}
 F(\tau) = -\sum_{m\in\Z^{2},m\ne0}\frac{m_{1}m_{2}}{\pi|m|^{4}}\,e^{im\cdot \tau}
  = \frac4\pi \sum_{m_{1},m_{2}=1}^{\infty}
    \frac{m_{1}m_{2}}{(m_{1}^{2}+m_{2}^{2})^{2}}
    \sin(m_{1}\tau_{1}) \sin(m_{2}\tau_{2})
        \,.
\end{equation}
Whereas we do not know an explicit closed form expression for this function, we know from the results of Section~\ref{SS:Ewald} using Ewald's method that it is bounded and that it can be written as in equation~\eqref{E:F=F0+Khat}
\begin{equation}
\label{E:F=F0+KhatEx2} 
 F(\tau) = \hat K(\tau) + F_{0}(\tau) \quad
  \mbox{ where $F_{0}$ is analytic in $Q$ and $F_{0}(0)=0$}.  
\end{equation}
In addition, we know from \eqref{E:FEx2} that $F$ vanishes on the boundary of $Q$, hence
\begin{equation}
\label{E:F0boundary}
 F_{0}(\tau) = -\hat K(\tau) \qquad (\tau\in\partial Q)\,.
\end{equation}
In the previous example, we used the explicit expression of $F(\tau)$ for finding the range of $F$. In fact, the function $F_{0}$ in that case was just the linear interpolation between the two values of the symbol $\hat K$ on $\partial Q$, which implied that the closed convex hull of $\im(F)$ was the same as the convex hull of $\im(\hat K)$. In the present case, we do not have a simple formula, but we can still prove that the conclusion is true.
\begin{lemma}
 \label{L:im(Fex2)}
 Let $F(\tau)$ be as defined in \eqref{E:FEx2}. Then for any $\tau\in Q$
\begin{equation}
\label{E:Fpos}
  F(\tau) \in \cC=[-\frac12,\frac12].
\end{equation}
\end{lemma}
The proof is not obvious, although the claim is numerically evident if we compute $F$ using Ewald's method and plot its graph, see the contour plot in Figure~\ref{F:ContEx2}.

\begin{figure}[h]
\centering
\includegraphics[width=0.6\textwidth]{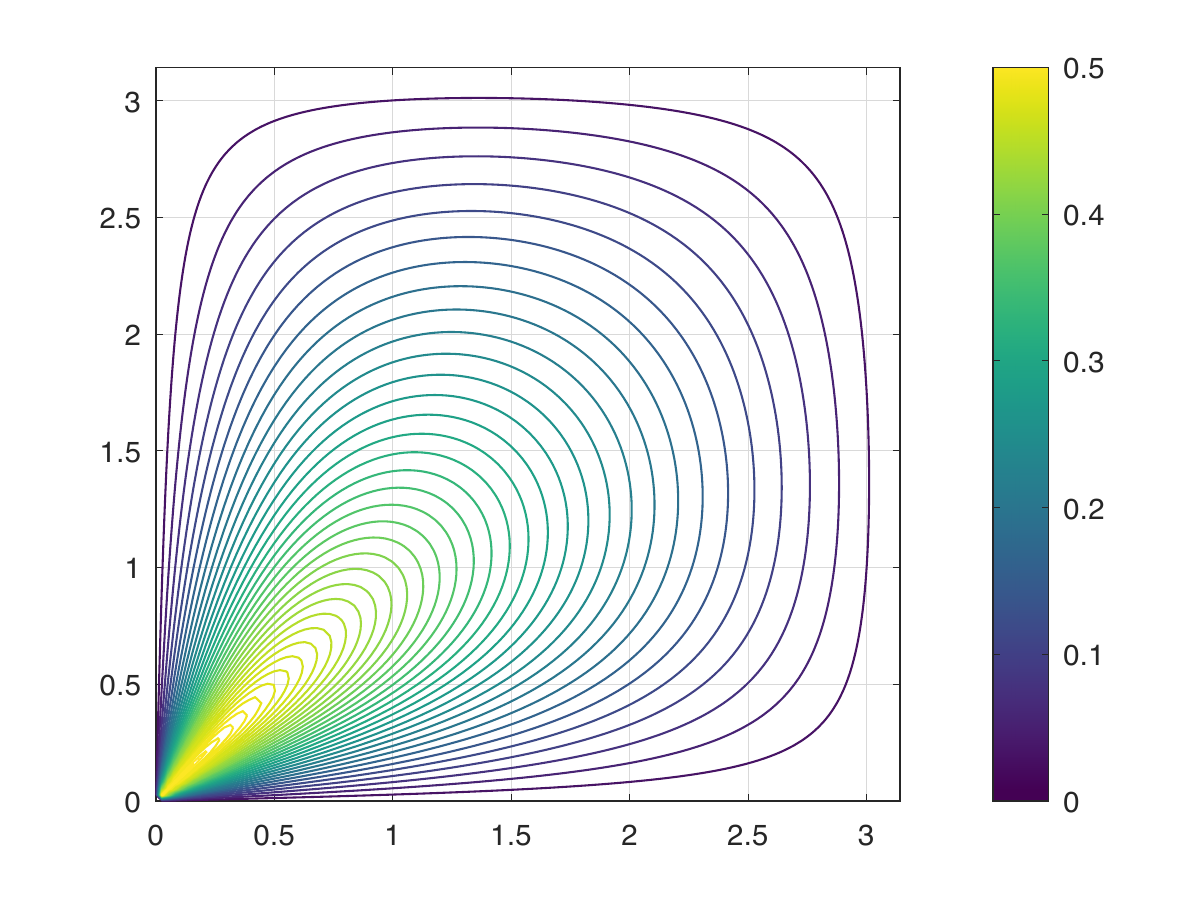}
\caption{Contour plot of $F(\xi)$ on the quarter square $Q_{++}$, Example \ref{Ex:2}.}
\label{F:ContEx2}
\end{figure}

Before we give the proof, let us draw the conclusion for the stability of the numerical scheme~\eqref{E:DDA2dEx2}.
\begin{corollary}
 \label{C:stabEX2}
 Let $\cC=[-\frac12,\frac12]$ and $\lambda\in\C\setminus\cC$. Then for any $N$ the linear system~\eqref{E:DDA2dEx2} has a unique solution, and there is a uniform resolvent estimate \begin{equation}
\label{E:ResTNex2}
 \|(\lambda\Id-T^{N})^{-1}\|_{\cL(\ell^{2}(\Omega^{N}))} \le \dist(\lambda,\cC)^{-1}\,.
\end{equation}
\end{corollary}

\begin{proof}
For symmetry reasons, it is sufficient to prove \eqref{E:Fpos} for $\tau\in Q_{++}=(0,\pi)^{2}$.
For the proof of Lemma~\ref{L:im(Fex2)}, we will show the following:
\begin{equation}
\label{E:FposF0neg}
\mbox{ For any }\xi\in Q_{++}\,,\quad
F(\xi)\ge0 \;\mbox{ and } \; F_{0}(\xi)\le0\,.
\end{equation}
This implies $0\le F(\xi)\le \hat K(\xi)\le\frac12$, hence \eqref{E:Fpos}\,.

We use the integral representations from Proposition~\ref{P:FInt}
\begin{equation}
\label{E:F0Int}
 F_{0}(\xi) = \int_{0}^{\infty} H_{0}(\xi,s)ds\,,\quad
 \hat K(\xi) = \int_{0}^{\infty} H_{1}(\xi,s)\,ds\,,\quad
 F(\xi) = \int_{0}^{\infty} H(\xi,s)\,ds\,.
\end{equation}
According to \eqref{E:H0}, $H_{0}(\xi,s)=H(\xi,s)-H_{1}(\xi,s)$ with
\begin{equation}
\label{E:H0H1Ex2}
 H_{1}(\xi,s)= \frac{\xi_{1}\xi_{2}}{4s^{2}} e^{-\frac{|\xi|^{2}}{4s}}\,,\quad
 H(\xi,s)=  \sum_{n\in\Z^{2}} H_{1}(\xi+2\pi n,s)\,.
\end{equation}
Let $0\le\epsilon<T$ and $\Sigma_{\epsilon}^{T}=Q_{++}\times (\epsilon,T)$.
In $\Sigma_{\epsilon}^{T}$, we want to use the maximum principle for the heat equation (see Lemma~\ref{L:heat}) for the functions 
$\tilde H_{0}(\xi,s)=s^{-1}H_{0}(\xi,s)$ and $\tilde H(\xi,s)=s^{-1}H(\xi,s)$. 

Since $\tilde H_{1}(\xi,s)=s^{-1}H_{1}(\xi,s)$ is continuous for $(\xi,s)\in\R^{2}\times[0,\infty)\setminus\{0,0\}$ and the Poisson series 
$$
 \tilde H_{0}(\xi,s) =  \sum_{n\in\Z^{2},n\ne0} \tilde H_{1}(\xi+2\pi n,s)
$$
converges uniformly for $(\xi,s)\in\overline{\Sigma_{0}^{T}}$ for all $T>0$, we see that 
$\tilde H_{0}$ is continuous in $\overline{\Sigma_{0}^{T}}$ with initial value 
$\tilde H_{0}(\xi,0)=0$. 
On the lateral boundary we use the Fourier representation (see \eqref{E:FFintsum})
$$
 \tilde H(\xi,s) = \frac4\pi \sum_{m_{1},m_{2}=1}^{\infty}
    m_{1}m_{2} e^{-|m|^{2}s}
    \sin (m_{1}\xi_{1}) \sin(m_{2}\xi_{2})\,.
$$
If $\xi_{1}$ or $\xi_{2}$ is in $\{0,\pi\}$, this implies that $\tilde H=0$ and therefore
$$
 \tilde H_{0}(\xi,s) = -\tilde H_{1}(\xi,s)\le0 
 \quad\mbox{ for }(\xi,s)\in\partial Q_{++}\times(0,T]\,.
$$
According to Lemma \ref{L:heat}, $\tilde H_{0}$ satisfies the heat equation
$(\partial_{s}-\Delta_{\xi})\tilde H_{0}=0$ in $\Sigma_{0}^{T}$.
Thus we can apply the maximum principle to $\tilde H_{0}$ and obtain 
$\tilde H_{0}(\xi,s)\le0$ in $\Sigma_{0}^{T}$, hence also $H_{0}(\xi,s)\le0$.
Integrating over $s\in(0,\infty)$ yields
$$
 F_{0}(\xi) \le 0 \quad\mbox{ for }\xi\in Q_{++}\,.
$$

For $\tilde H$, we cannot apply the maximum principle directly in $\Sigma_{0}^{T}$, because $\tilde H$ is not continuous at $(0,0)\in\overline{\Sigma_{0}^{T}}$, but we can apply it in
$\Sigma_{\epsilon}^{T}$ for any $0<\epsilon<T$. On the lateral boundary, $\tilde H$ vanishes as seen above, and for the initial value at $s=\epsilon$ we have
$$
 \tilde H(\xi,\epsilon) = \tilde H_{0}(\xi,\epsilon) + \tilde H_{1}(\xi,\epsilon)
 \ge \tilde H_{0}(\xi,\epsilon) \ge \delta(\epsilon)
$$
with $\delta(\epsilon)=\inf_{\xi\in Q_{++}}\tilde H_{0}(\xi,\epsilon)$.
Hence by the maximum principle, in $\overline{\Sigma_{\epsilon}^{T}}$ we have 
$$
 \tilde H(\xi,s) \ge \min\{0,\delta(\epsilon)\}\,.
$$
Now, as we have seen above, $\tilde H_{0}(\cdot,s)$ tends to $0$ uniformly as $s\to0$, hence $\delta(\epsilon)\to0$ as $\epsilon\to0$, which implies $\tilde H(\xi,s)\ge0$ for any $s>0$ and $\xi\in Q_{++}$. After integrating over $s$, we finally get $F(\xi)\ge0$ for $\xi\in Q_{++}$, and the proof of the Lemma is complete.
\end{proof}

\begin{remark}
 \label{R:idealEx2}
 In conclusion, for this example we find the same ``ideal'' stability estimate as in the previous one-dimensional example.
\end{remark}

\refstepcounter{Ex}\label{Ex:3}
%%%%%%%%%
\subsection{Example 3. Dimension $d=2$, kernel $(x_{1}^{2}-x_{2}^{2})|x|^{-4}$}\label{SS:2da12}\ 
%%%%%%%%%

We consider another two-dimensional example where in the notation of Section~\ref{SSS:Specific} $d=2$ and $p(x)=-\frac1{2\pi} a_{12}(x)$, see \eqref{E:ajk,bjk}. 
We show that the complement of the stability zone in this case is strictly larger than the image of the symbol of the integral operator.

%%%%%%%%%
\subsubsection{The singular integral equation}\label{SSS:SIOEx3}
%%%%%%%%%
We use the same notation for analogous objects as in the preceding example. Therefore in this section, the letters $K$, $\hat K$, $T$ etc.\ are redefined to have new meanings.
The kernel and its Fourier transform are now
\begin{equation}
\label{E:KandKhatfora}
 K(x) = \frac{x_{2}^{2}-x_{1}^{2}}{2\pi|x|^{4}}\,,\qquad
 \hat K(\xi) = \frac{\xi_{1}^{2}-\xi_{2}^{2}}{2|\xi|^{2}}
   = \frac{\xi_{1}^{2}}{|\xi|^{2}} - \frac12\,.
\end{equation}
The normalization is chosen so that the range of the function $\hat K$ is again the interval 
$[-\frac12,\frac12]$. 

In fact, this kernel is the same as in the previous example \eqref{E:KandKhatforb} after a $45^{\circ}$ rotation of the coordinate system. 
Therefore if we write the singular integral equation as in \eqref{E:IEEx2}, we can copy verbatim the statement of the previous example concerning the numerical range of the integral operator $A_{\Omega}$ (see Lemma~\ref{L:AOmegaEx2}) and the resolvent estimate \eqref{E:resAOmega2DEx2}.
\begin{lemma}
 \label{L:AOmegaEx3}
 Lemma \ref{L:AOmegaEx2} is true for the singular integral equation~\eqref{E:IEEx2} defined with the kernel~\eqref{E:KandKhatfora}.
\end{lemma}

%%%%%%%%%%
\subsubsection{The discrete system}\label{SS:DA2dEx3}
%%%%%%%%%
To the delta-delta discretization 
\begin{equation}
\label{E:DDA2dEx3}
 \lambda u_{m} - N^{-2}\sum_{n\in\omega^{N},m\ne n} K(x^{N}_{m}-x^{N}_{n}) u_{n}
   = f(x^{N}_{m}) \quad (m\in\omega^{N})
\end{equation}
corresponds the finite section $T^{N}$ of the infinite Toeplitz matrix
$$
  T = \frac1{2\pi} \Big(  
  \frac{(m_{2}-n_{2})^{2}-(m_{1}-n_{1})^{2}}{|m-n|^{4}}
  \Big)_{m,n\in\Z^{2}}\quad
  \mbox{ with zero on the diagonal.}
$$ 
The numerical symbol (symbol of $T$) is now defined as
\begin{equation}
\label{E:FEx3}
 F(\tau) = \sum_{m\in\Z^{2},m\ne0}\frac{m_{2}^{2}-m_{1}^{2}}{2\pi|m|^{4}}\,e^{im\cdot \tau}
 \,.
\end{equation}
\begin{lemma}
 \label{L:im(Fex3)}
 Let $$\Lambda_{0}=\dfrac{\Gamma(\frac14)^{4}}{32\pi^{2}}=0.5471...\,.$$ 
 Let $F(\tau)$ be as defined in \eqref{E:FEx3}. 
 Then there exists $\Lambda_{+}\ge\Lambda_{0}$ such that 
 $F(Q)=\cC=[-\Lambda_{+},\Lambda_{+}]$. 
 \end{lemma}
\begin{conjecture}
 \label{C:Lambda+=Lambda0}
 Numerical evidence suggests equality 
 \begin{equation}
\label{E:Lambda+=Lambda0}
 \Lambda_{+}=\Lambda_{0}\,.
\end{equation}
\end{conjecture}
%\begin{proof}
\par\noindent{\it Proof of Lemma~\ref{L:im(Fex3)}}. \ignorespaces
 The function $F$ is odd with respect to permutation of $\xi_{1}$ and $\xi_{2}$. The decomposition $F=F_{0}+\hat K$ with $F_{0}$ continuous on $Q$ implies that $F$ takes its maximum $\Lambda_{+}$ on $Q$. Therefore its image $F(Q)$ is a closed symmetric interval 
$\cC=[-\Lambda_{+},\Lambda_{+}]$. We are going to show that 
\begin{equation}
\label{E:F(pi)=Lambda0}
 F(\pi,0) = \Lambda_{0}\,.
\end{equation}
The conjecture \eqref{E:Lambda+=Lambda0} then corresponds to the claim that $F$ attains its maximum on $Q$ in the point $\tau=(\pi,0)$.

To prove \eqref{E:F(pi)=Lambda0}, we first transform the slowly converging double Fourier series
\begin{equation}
\label{E:FLambda0}
 F(\pi,0) = 
  \sum_{m\in\Z^{2},m\ne0}(-1)^{m_{1}}\frac{m_{2}^{2}-m_{1}^{2}}{2\pi(m_{1}^{2}+m_{2}^{2})^{2}}
\end{equation}
into a rapidly convergent single series. One way to get this is to start with the Poisson summation formula applied to the function $f(x)=(x-iy)^{-1}$ whose Fourier transform is 
$\hat f(\xi)=2\pi i \mathds{1}_{+}(\xi)e^{-y\xi}$ for $y>0$. The result is then valid for all $y\ne0$. It can be written for $t\in[-\pi,\pi]$ as
\begin{equation}
 \label{E:1Ex3}
 \sum_{n\in\Z}\frac{e^{int}}{n-iy} =
 i \pi \frac{e^{y\sigma(t)}}{\sinh(\pi y)} 
 \quad \mbox{ with }
 \sigma(t) = -t + \pi\,\sign\, t\,.
\end{equation}
Remark: Euler's formula \eqref{E:sawtooth} is a simple consequence of this.\\
Taking the derivative with respect to $y$ and subtracting the formulas for $y$ and $-y$ leads to
\begin{equation}
 \label{E:sum1Ex3}
 \sum_{n\in\Z}\frac{n^{2}-y^{2}}{(n^{2}+y^{2})^{2}}e^{int} =
 \pi \,
  \frac{\sigma(t)\sinh(\sigma(t)y)\sinh(\pi y)-\pi\cosh(\sigma(t)y)\cosh(\pi y)}{\sinh^{2}\pi y} .
 \end{equation}
This can be used to reduce the double Fourier series for $F(\xi)$ to a single rapidly convergent Fourier series. We are here only interested in the limit $t\to0$:
\begin{equation}
 \label{E:sum0Ex3}
 \sum_{n\in\Z}\frac{n^{2}-y^{2}}{(n^{2}+y^{2})^{2}} =
  \frac{-\pi^{2}}{\sinh^{2}\pi y} \,.
 \end{equation}
Hence, by decomposing the double sum $\sum_{m\in\Z^{2}\setminus\{0\}}$ as
$\sum_{m_{1}=0,m_{2}\in\Z\setminus\{0\}}+\sum_{m_{1}\in\Z\setminus\{0\},m_{2}\in\Z}$ and using
$\sum_{n=1}^{\infty}\frac1{n^{2}}=\frac{\pi^{2}}6$ (which can also be obtained from \eqref{E:sum0Ex3} by looking at the pole in $y=0$)
we finally get
\begin{equation}
\label{E:Lambda0}
\Lambda_{0}=
\sum_{m\in\Z^{2},m\ne0}(-1)^{m_{1}}\frac{m_{2}^{2}-m_{1}^{2}}{2\pi(m_{1}^{2}+m_{2}^{2})^{2}}
=  \frac\pi6 - \sum_{n=1}^{\infty}\frac{(-1)^{n}\pi}{\sinh^{2}\pi n}\,.
\end{equation}
This series converges rapidly, with $5$ terms giving $15$ significant digits:
$\Lambda_{0}=0.547109903806619...$
The series is covered by the formulas for $\mathrm{IX}_{s}$ with $s=2$ in \cite{Zucker1979}.
The explicit expression for $\Lambda_{0}$ given in the Lemma can be deduced from this.
\qed\endtrivlist
%\end{proof}

\begin{remark}
 \label{R:confrad}
 The conformal radius (or logarithmic capacity) of the unit square is known to be \cite[Tables]{PolyaSzego1951}
 %% better citation needed
 $$
   R_{\Box}= \frac{\Gamma(\frac14)^{2}}{4\pi^{\frac32}}\,.
 $$
 This implies the remarkable relation
 \begin{equation}
\label{E:ConfArea}
  \pi R_{\Box}^{2}= 2 \Lambda_{0} \,.
\end{equation}
%Thus the delta-delta discretization magnifies the spectrum of the singular integral operator by a factor that coincides with what one could call the conformal area of the unit square.
\end{remark}
The conjecture that $\Lambda_{+}=\Lambda_{0}$ is clearly supported by numerical evidence. Here are the results of two different approaches for the approximation of $\Lambda_{+}$: 

In Table~\ref{T:supF} we approximate the numerical symbol $F$ from \eqref{E:FEx3} using the Ewald method \eqref{E:Fewald} from Proposition~\ref{P:NumSym}.
\begin{equation}
\label{E:FewaldEx3}
 F(\tau) \approx 
 \sum_{|m_{1}|,|m_{2}|\le M,m\ne0}K(m)\Gamma(2,\pi|m|^{2})e^{im\cdot \tau} +
 \sum_{|n_{1}|,|n_{2}|\le M}\hat K(\tau+2\pi n)\,e^{-\frac{|\tau+2\pi n|^{2}}{4\pi}}\,.
\end{equation}
We take the maximum of $F(\tau)$ over a regular $N\times N$ grid discretizing the period square $Q=[-\pi,\pi]^{2}$. Results are shown for $N=1001$, so that the point $(\pi,0)$ is included. One sees the rapid convergence of the sums in the Ewald method. 
\begin{table}[htbp]
   \centering
    \begin{tabular}{@{} lcr @{}} 
      \toprule
      $M$    & Maximum & diff with $\Lambda_{0}$\\
      \midrule
1 & 0.5466820485568409 & -0.00043 \\
2 & 0.5471099022284376 & -1.578e-9\\
3 & 0.5471099038066192 & 1.11e-16\\
4 & 0.5471099038066192 & 1.11e-16\\
      \bottomrule
   \end{tabular}\vglue1ex
   \caption{Computation of $\Lambda_{+}$}
   \label{T:supF}
\end{table}

In Table~\ref{T:maxEv} we show the maximum eigenvalue of the matrix $T^{N}$ where $\Omega$ is the unit square, together with an extrapolated value and its difference with $\Lambda_{0}$. 
\begin{table}[htbp]
   \centering
    \begin{tabular}{@{} lccr @{}} % Column formatting, @{} suppresses leading/trailing space
      \toprule
      $N$    &  $\lambda_{\max}(T^{N})$ & extrap. & diff with $\Lambda_{0}$\\
      \midrule
16 & 0.541802946417726 &  &\\
24 & 0.544571778645890 &  & \\
36 & 0.545922219922679 &  0.547207966733364 & 9.81e-5\\
54 & 0.546562896841136 & 0.547141211191569 & 3.13e-5\\
81 & 0.546860792009930 & 0.547119678405314 & 9.77e-6\\
      \bottomrule
   \end{tabular}\vglue1ex
   \caption{Computation of $\max(\Sp(T^{N}))$, Example \ref{Ex:3}}
   \label{T:maxEv}
\end{table}

For comparison, we show in Table~\ref{T:maxEvEx2} the analogous computations for the matrices from Example \ref{Ex:2}, where $\Lambda_{+}=0.5$.
\begin{table}[htbp]
   \centering
    \begin{tabular}{@{} lccr @{}} % Column formatting, @{} suppresses leading/trailing space
      \toprule
      $N$    &  $\lambda_{\max}(T^{N})$ & extrap. & diff with $0.5$\\
      \midrule
16 & 0.4299869696672885 &  &\\
24 & 0.4526591158216325 &  & \\
36 & 0.4683227545642122 &  0.5033301483277116 & 3.33e-3\\
54 & 0.4789372344435390 & 0.5012512843991882 & 1.25e-3\\
81 & 0.4860451011088278 & 0.5004526691660266 & 4.53e-4\\
      \bottomrule
   \end{tabular}\vglue1ex
   \caption{Computation of $\max(\Sp(T^{N}))$, Example \ref{Ex:2}}
   \label{T:maxEvEx2}
\end{table}

In the previous Example \ref{Ex:2}, we were able to prove the equation $\Lambda_{+}=0.5$ using an argument involving the maximum principle for the heat equation, see the proof of Lemma~\ref{L:im(Fex2)}, in particular \eqref{E:FposF0neg}.
While we have no proof for the equation $\Lambda_{+}=\Lambda_{0}$ here,
it is possible to use an analogous argument to obtain an upper bound for $\Lambda_{+}$. The square $Q_{++}$ (of area $\pi^{2}$) of the previous example now has to be turned by $45^{\circ}$ and to be replaced by the lozenge (a square of area $2\pi^{2})$
$$
  Q_{\lozenge} = \big\{\xi\in\R^{2} \mid 0<\xi_{1}<2\pi;\; |\xi_{2}|<\min\{\xi_{1}, 2\pi-\xi_{1}\}\big\}. 
$$ 
Then one can see again that $F(\xi)=0$ on $\partial Q_{\lozenge}$. 
But now $\partial Q_{\lozenge}$ contains two points of discontinuity of $F$, the origin $(0,0)$ and the point $(2\pi,0)$. Therefore the decomposition $F=F_{0}+\hat K$ has to be refined into
$$
  F(\xi) = F_{00}(\xi) + \hat K(\xi) + \hat K(\xi-(2\pi,0))\,.
$$
The function $F_{00}$ defined by this will then be continuous on the closure of $Q_{\lozenge}$. Now one can use the integral representation from Proposition~\ref{P:FInt} similarly to \eqref{E:F0Int} and use the maximum principle for the heat equation as before to conclude that
$$
 \mbox{ For any }\xi\in Q_{\lozenge}\,,\quad
F(\xi)\ge0 \;\mbox{ and } \; F_{00}(\xi)\le0\,.
$$
This implies $0\le F(\xi)\le \hat K(\xi)+ \hat K(\xi-(2\pi,0))$ in $Q_{\lozenge}$, and hence by taking the maximum,
\begin{equation}
\label{E:Lambda+Ex3}
\Lambda_{+}\le 1\,.
\end{equation}
Unfortunately, this estimate is much less sharp than the estimate by $\frac12$ in the previous example.

To illustrate the behavior of the numerical symbol $F(\xi)$, we present in Figure~\ref{F:NumSymEx3} a surface graph of $F$ on the square $Q$ and a contour plot on the lozenge $Q_{\lozenge}$. The parts exceeding the range of the symbol $\hat K$ are indicated in bright red hues. The maximum at the midpoint $(\pi,0)$ of $Q_{\lozenge}$ is in clear evidence.
\begin{figure}[h]
\centering
\includegraphics[width=0.495\textwidth]{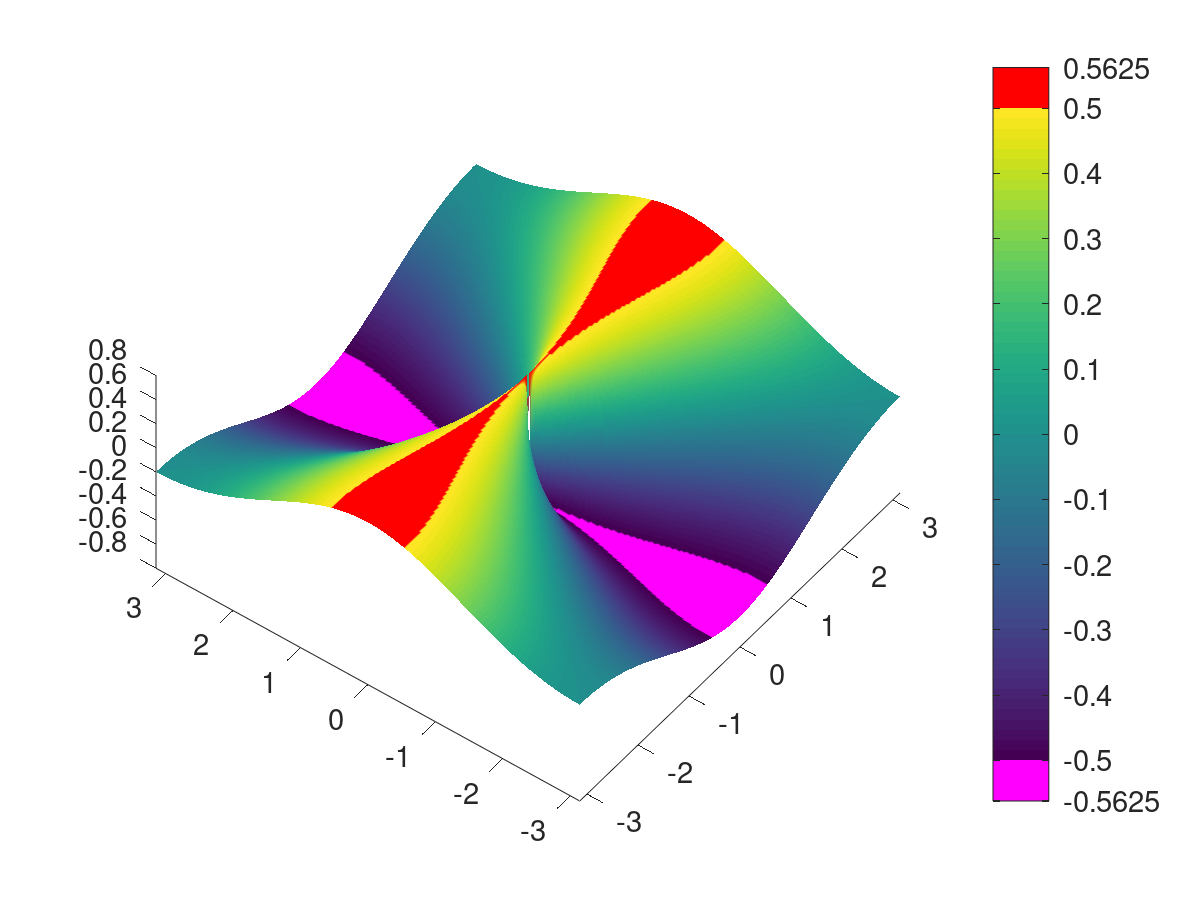}
\includegraphics[width=0.495\textwidth]{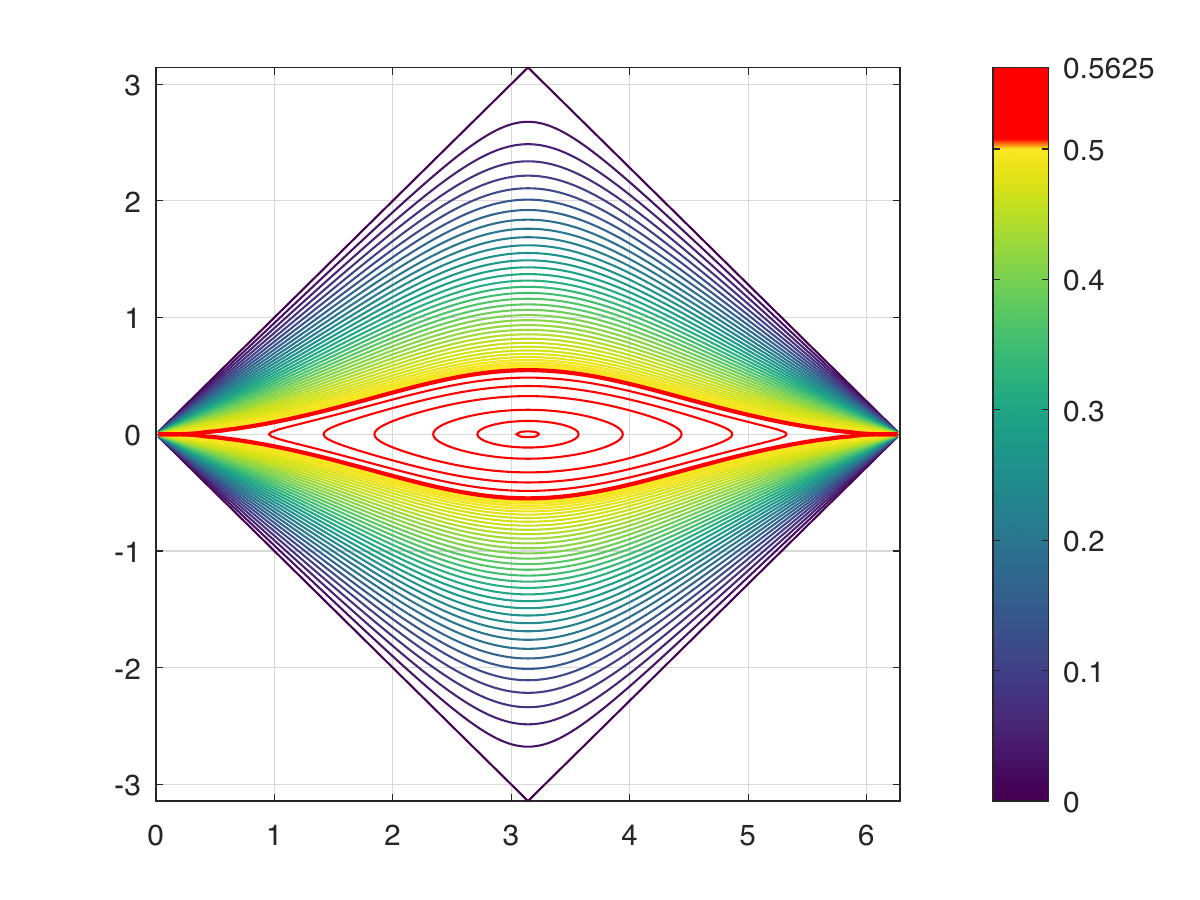}
\caption{Numerical symbol for Example \ref{Ex:3}.
Left: $F$ on $Q$, right: $F$ on $Q_{\lozenge}$.}
\label{F:NumSymEx3}
\end{figure}

Let us summarize the stability result obtained for this example.

\begin{corollary}
 \label{C:stabEX3}
 Let $\cC=[-\Lambda_{+},\Lambda_{+}]$ and $\lambda\in\C\setminus\cC$. Then for any $N$ the linear system~\eqref{E:DDA2dEx2} has a unique solution, and there is a uniform resolvent estimate 
\begin{equation}
\label{E:ResTNex3}
 \|(\lambda\Id-T^{N})^{-1}\|_{\cL(\ell^{2}(\Omega^{N}))} \le \dist(\lambda,\cC)^{-1}\,.
\end{equation}
For $\lambda\in[-\Lambda_{+},-\frac12)\cup(\frac12,\Lambda_{+}]$ the integral equation 
$(\lambda \Id -A_{\Omega})u=f$ with kernel \eqref{E:KandKhatfora} is well-posed in $L^{2}(\Omega)$, but the corresponding delta-delta approximation scheme \eqref{E:DDA2dEx3} is unstable.
\end{corollary}

\refstepcounter{Ex}\label{Ex:4}
%%%%%%%%%
\subsection{Example 4. Dimension $d=2$, kernel $(x_{1}+ix_{2})^{2}|x|^{-4}$}\label{SS:2da+ib}\ %%%%%%%%%

Let $d=2$ and $p(x)=-\frac1{1\pi} (a_{12}(x)+2ib_{12}(x))$. The corresponding kernel and its Fourier transform are
\begin{equation}
\label{E:KandKhatforEx4}
 K(x) = \frac{x_{2}^{2}-x_{1}^{2}-2ix_{1}x_{2}}{\pi|x|^{4}}\,,\qquad
 \hat K(\xi) = \frac{(\xi_{1}+i\xi_{2})^{2}}{|\xi_{1}+i\xi_{2}|^{2}}.
\end{equation}
The normalization is chosen so that $|\hat K(\xi)|=1$ for $\xi\in\R^{2}$.
We include this example, which has features combining those of the two preceding examples, mainly for purposes of illustration. Because the singular integral operator and the system matrices of the corresponding delta-delta discretization in this case are non-selfadjoint, we expect to see less trivial relations between spectra and numerical ranges than in the selfadjoint case.

It is obvious from the definition \eqref{E:KandKhatforEx4} that the spectrum $\Sp(A)$ of the operator of convolution with $K$ in $L^{2}(\R^{2})$ is the unit circle 
$\{\xi\in\C\mid |\xi|=1\}$ and that its numerical range is the unit disk. Whereas we do not know the spectrum $\Sp(A_{\Omega})$ for a bounded domain $\Omega\subset\R^{2}$, the numerical range is still the unit disk, compare \eqref{E:WAOmega=WA}, 
\begin{equation}
\label{E:SpW(AOmega)Ex4}
 \Sp(A_{\Omega})\subset W(A_{\Omega}) = W(A) = \{\xi\in\C\mid |\xi|\le1\}.
\end{equation}

For the system matrices $T^{N}$ of the delta-delta discretization scheme, Theorem~\ref{T:stabgeneral} and Lemma~\ref{L:TbyF} provide the following relations.
\begin{equation}
\label{E:SpW(TN)Ex4}
 \Sp(T^{N})\subset W(T^{N}) \subset \overline{W(T)} = \overline{\mathop{\rm conv}}\bigcup_{M\in\N}W(T^{M})\quad\mbox{ and }\quad
 W(A) \subset \overline{W(T)}\,.
\end{equation}

\begin{figure}[h]
\centering
\includegraphics[width=0.49\textwidth]{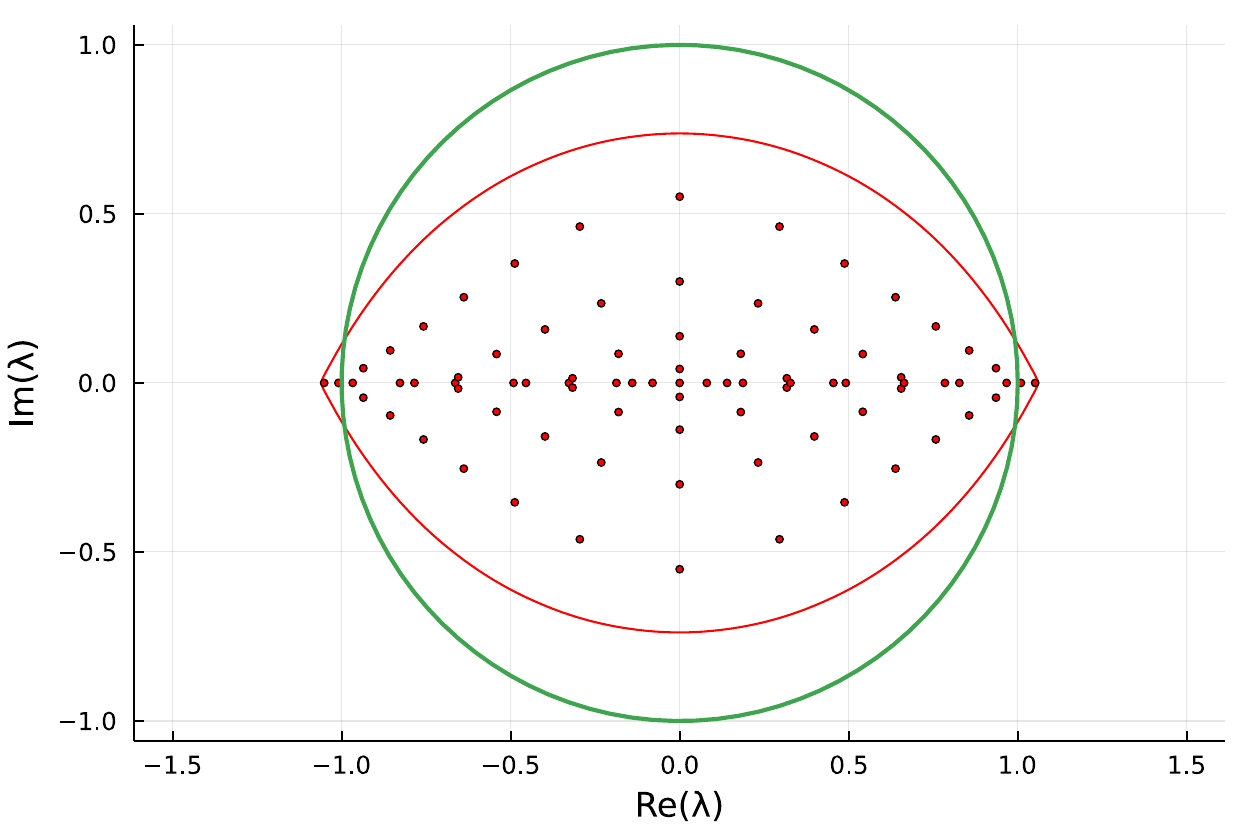}
\includegraphics[width=0.49\textwidth]{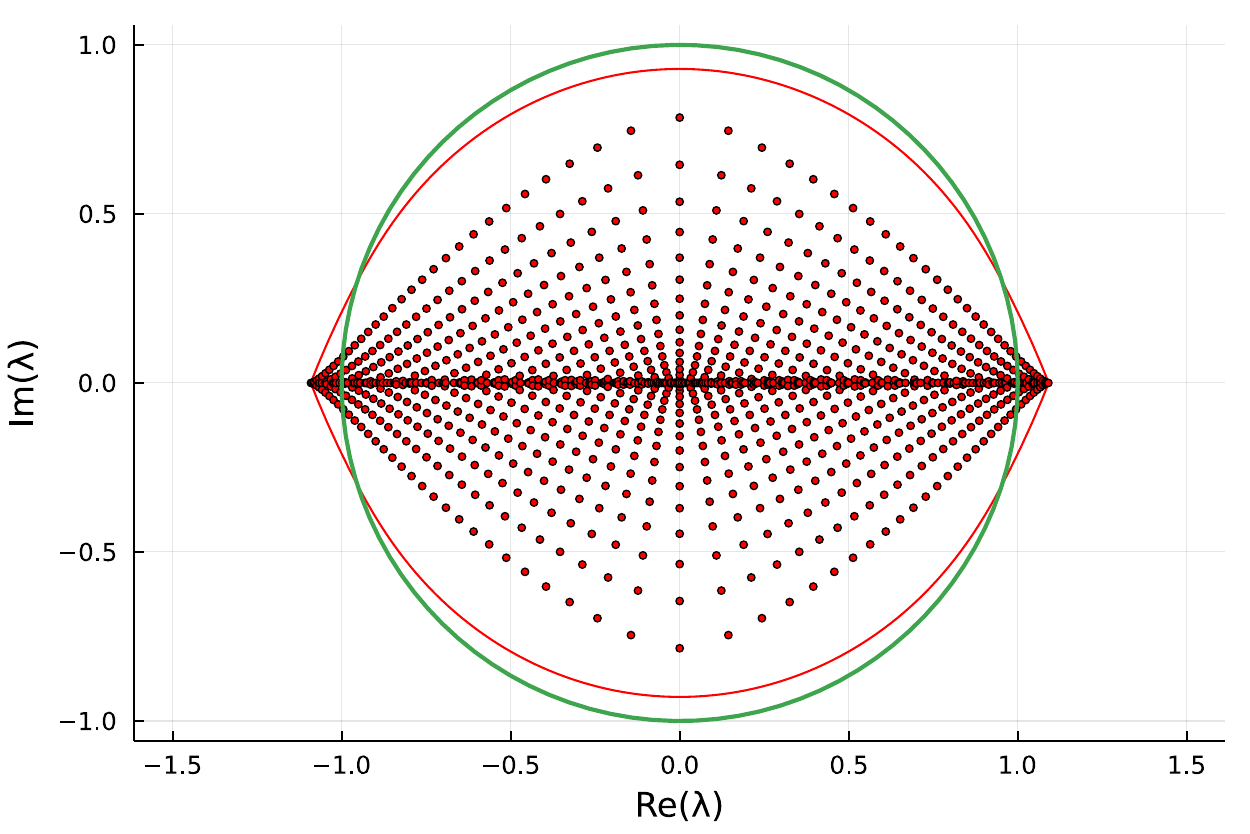}
\caption{Spectrum and numerical range.
Left: $N=8$, right: $N=32$.}
\label{F:Ex4}
\end{figure}

In Figure~\ref{F:Ex4} we show for the case of a square domain $\Omega$ and two values of $N$ the spectrum $\Sp(T^{N})$ (red points), the boundary of the numerical range $W(T^{N})$ (red line), and the unit circle, which is the boundary of $W(A)$ (green line).
We can see the inclusions from \eqref{E:SpW(TN)Ex4} between $\Sp(T^{N})$ and $W(T^{N})$, and we can perceive the asymptotic inclusion of $W(A)$ in $W(T^{N})$ as $N$ tends to infinity.

We can also see that the eigenvalues of the matrices, in contrast to the numerical range, will not fill the whole unit disk asymptotically. On the other hand, we clearly see the overshoot 
$W(T)\setminus W(A_{\Omega})$, that is the region of $\lambda$ where the volume integral equation is uniquely solvable and the operator $\lambda\Id-A_{\Omega}$ is sectorial, so that every $L^{2}$-conforming Galerkin method would converge, whereas the delta-delta scheme is unstable. It appears that the limits for the real part of this overshoot are the same (scaled by a factor $2$) as in the previous example, that is 
$\pm2\Lambda_{0}$ with $\Lambda_{0}$ defined in Lemma~\ref{L:im(Fex3)}.

\refstepcounter{Ex}\label{Ex:5}
%%%%%%%%%
\subsection{Example 5. Dimension $d\ge2$. Volume Integral Equation for the Quasi-static Maxwell system}\label{SS:QSM}\ 
%%%%%%%%%

In the quasi-static Maxwell volume integral equation (see Section~\ref{SS:DDA}), the right hand side and the solution are $\C^{d}$-valued functions, and the singular integral operator is defined as the matrix of second distributional derivatives of the convolution with the free-space Green function $g$ for the Laplace operator, see equation \eqref{E:VIEdist}. 
If we call this operator $A^{0}$, then it is not the same as the operator $A$ defined by the Cauchy principal value of the integral with the same kernel, but there is a simple relation:
Let 
\begin{equation}
\label{E:KmaxQS}
K(x)= -D^{2}g(x) = \big(-\partial_{i}\partial_{j}g(x)\big)_{i,j\in1,\dots,d} \quad\mbox{ for } x\ne0\,.
\end{equation}
Then
\begin{equation}
\label{E:D2VSPV}
A^{0}u(x) = -\nabla\div\int_{\R^{d}}g(x-y)u(y)dy =
 \pv\!\!\int_{\R^{d}} K(x-y)u(y)dy  + \frac1d u(x) =
 (A+\frac1d\Id)u(x)\,.
\end{equation}
This is most easily seen by first using the symmetries of the kernel with respect to reflections at coordinate axes and permutations of the variables in order to deduce that the distribution kernel of $A^{0}-A$ must be a scalar multiple of the $d\times d$ identity matrix 
$\Id_{d}$ times the Dirac distribution $\delta_{0}$, and then determining this multiple by taking traces:
$\mathop{\mathrm{tr}}(-D^{2}g)=-\Delta g =\delta_{0}= \mathop{\mathrm{tr}}(\frac1d \Id_{d}\delta_{0})$.

%%%%%%%%%
\subsubsection{The singular integral equation}\label{SSS:SIOEx5}
%%%%%%%%%
We consider the strongly singular integral equation, still written as 
$(\lambda \Id -A_{\Omega})u=f$,
\begin{equation}
\label{E:VIOMQS}
 \lambda u(x) - \pv\!\! \int_{\Omega}K(x-y) u(y)\,dy = f(x) \quad
 \mbox{ with } K \mbox{ given in \eqref{E:KmaxQS}}\,.
\end{equation}
The function space is now $L^{2}(\Omega;\C^{d})$.

Let us note the explicit form of the kernel, valid in any dimension $d\ge2$, where we consider points in $\R^{d}$ as column vectors,
\begin{equation}
\label{E:Kexplicit}
 K(x) = -\frac1{\nu_{d}} (x\,x^{\top}-\frac1d\Id_{d}|x|^{2})\,|x|^{-d-2}\,,
  \quad\mbox{ with }\;
  \nu_{d} = \frac{2\pi^{\frac d2}}{d\,\Gamma(\frac d2)}\,.\end{equation}
The simplest way to see this is to first look at the symbol of the operator. For this we employ $d$-dimensional Fourier transformation and use the fact that $\hat g(\xi)=|\xi|^{-2}$, hence
\begin{equation}
\label{E:SymMQS}
 \sF(-D^{2}g)(\xi) = \frac{\xi\,\xi^{\top}}{|\xi|^{2}} \quad\mbox{ and }\;
 \hat K(\xi) = \frac{\xi\,\xi^{\top}-\frac1d\Id_{d}|\xi|^{2}}{|\xi|^{2}}\,.
\end{equation}
We check that $\mathop{\mathrm{tr}}\hat K=0$ and that $\hat K$ satisfies the spherical cancellation condition. Indeed, in the notation of Lemma~\ref{L:Khat}, the (matrix-valued) polynomial $p(\xi)$ is given by
\begin{equation}
\label{E:pxiMQS}
  p(\xi)= -\frac1{\nu_{d}} (\xi\,\xi^{\top}-\frac1d\Id_{d}|\xi|^{2})\,.
\end{equation}
Thus the off-diagonal elements of the matrix $p(x)$ are given by
$$
  -\frac1{\nu_{d}}x_{j}x_{k} = -\frac{b_{jk}(x)}{\nu_{d}}, 
$$ 
and the diagonal elements by
$$
 -\frac1{\nu_{d}}(x_{k}^{2}-\frac1d|x|^{2})= \frac1{d\,\nu_{d}}\sum_{j=1}^{d}a_{jk}(x)\,,
$$
compare Remark~\ref{R:simple}.

From our formulas of Section~\ref{SSS:Homog} we find the explicit form \eqref{E:Kexplicit} for our kernel. For $d=2$, we have $\nu_{d}=\pi$, and we recognize the kernels studied in the Examples~\ref{Ex:2} and \ref{Ex:3}.

The matrix $\frac{\xi\,\xi^{\top}}{|\xi|^{2}}$ is an orthogonal projection matrix, hence its numerical range is the interval $[0,1]$. Therefore $W(\hat K(\xi))=[-\frac1d,1-\frac1d]$ for any $\xi\ne0$. We immediately get the following instance of Proposition~\ref{P:W(AOmega)}.
\begin{lemma}
 \label{L:WAMQS}
 Let $\cC=[-\frac1d,1-\frac1d]$. Then for all $\lambda\not\in\cC$ and any $f\in L^{2}(\Omega;\C^{d})$, the integral equation \eqref{E:VIOMQS} has a unique solution $u\in L^{2}(\Omega;\C^{d})$, and there is a resolvent estimate in the $L^{2}(\Omega;\C^{d})$ operator norm
\begin{equation}
\label{E:resAOmegaMQS}
 \|(\lambda\Id-A_{\Omega})^{-1}\| \le \dist(\lambda,\cC)^{-1}\,.
\end{equation} 
\end{lemma}

%%%%%%%%%%
\subsubsection{The discrete system}\label{SS:DA2dEx5}
%%%%%%%%%
With the $d\times d$ matrix-valued kernel $K$ and vector-valued functions $u$ and $f$, we can write the delta-delta discretization $(\lambda\Id - T^{N})U=F$ of the integral equation \eqref{E:VIOMQS} in the same form as in the scalar case
\begin{equation}
\label{E:DDAEx5}
 \lambda u_{m} - N^{-d}\!\!\!\!\sum_{n\in\omega^{N},m\ne n}K(x^{N}_{m}-x^{N}_{n})u_{n}= f(x^{N}_{m}) \,,\quad (m\in\omega^{N})\,,
\end{equation}
where now the system matrix $T^{N}$ is of size $d|\omega^{N}|\times d|\omega^{N}|$ and is considered as a linear operator in $\ell^{2}(\omega^{N};\C^{d})$.

We recall the discussion of matrix-valued kernels in Section~\ref{SS:matrix} above, in particular the properties of the numerical range stated in Lemma~\ref{L:TbyFEx5}.

The basic stability estimate follows.
\begin{proposition}
 \label{P:W(F)Ex5}
Let $K$ be the kernel defined in \eqref{E:KmaxQS}, \eqref{E:Kexplicit}. Then there exist 
$\Lambda_{-}^{(d)},\Lambda_{+}^{(d)}\in\R$ with 
\begin{equation}
\label{E:Lambda+-Ex5}
\Lambda_{-}^{(d)}\le-\frac1d\,,\qquad \Lambda_{+}^{(d)}\ge 1-\frac1d
\end{equation}
such that the following holds.\\
(i)
For $\tau\in Q=[-\pi,\pi]^{d}$, $\tau\ne0$, $F(\tau)$ is a real symmetric matrix with eigenvalues contained in the interval $\cC=[\Lambda_{-}^{(d)},\Lambda_{+}^{(d)}]$,
\begin{equation}
\label{E:Lambda+-SpF}
\Lambda_{-}^{(d)} = \inf_{\tau\in Q}\min(\Sp(F(\tau))\,,\qquad
\Lambda_{+}^{(d)} = \sup_{\tau\in Q}\max(\Sp(F(\tau))\,.
\end{equation}
(ii) 
For any $N$, the numerical range $W(T^{N})$ is contained in $W(T)=\cC$.\\
\begin{equation}
\label{E:Lambda+-SpT}
\Lambda_{-}^{(d)} = \inf_{N\in\N}\min(\Sp(T^{N}))\,,\qquad
\Lambda_{+}^{(d)} = \sup_{N\in\N}\max(\Sp(T^{N}))\,.
\end{equation}(iii)
The delta-delta scheme \eqref{E:DDAEx5} is stable if and only if $\lambda\in\C\setminus\cC$,
and one has the stability estimate in the $\ell^{2}(\omega^{N};\C^{d})$ operator norm
\begin{equation}
\label{E:ResTNex5}
 \|(\lambda\Id-T^{N})^{-1}\| \le \dist(\lambda,\cC)^{-1}\,.
\end{equation}
\end{proposition}
\begin{proof}
The matrix $K(x)$ is symmetric for $x\ne0$, implying that also $F(\tau)$ is a symmetric matrix for $\tau\ne0$. The symmetry $K(-x)=K(x)$ implies that the matrix elements of $F(\tau)$ are real.
Therefore the numerical range of $F(\tau)$ is the interval $[\lambda_{-}(\tau),\lambda_{+}(\tau)]$, where 
$$
  \lambda_{-}(\tau) = \min(\Sp(F(\tau))\,,\qquad
\lambda_{+}(\tau) = \max(\Sp(F(\tau))\,.
$$
This justifies \eqref{E:Lambda+-SpF}.
All the other statements of the proposition are instances of the statements of Section~\ref{SS:DDD}, in particular Theorem~\ref{T:stabgeneral}, and their proofs in Section~\ref{S:discrete}, based on Ewald's method.
\end{proof}
What remains is to get information on the numbers $\Lambda_{\pm}^{(d)}$ and to see whether the inequalities \eqref{E:Lambda+-Ex5} are strict. In that case,
for $\lambda\in[\Lambda_{-}^{(d)},-\frac1d)\cup(1-\frac1d,\Lambda_{+}^{(d)}]$,
the integral equation is well-posed in $L^{2}(\Omega;\C^{d})$, but the delta-delta discretization scheme is unstable in $\ell^{2}(\Omega^{N};\C^{d})$.

We will discuss this for the practically relevant cases $d=2$, where we get rather precise information, and $d=3$, which is the most important case because of its relevance for the DDA method in computational electromagnetics.
  
%%%%%%%%%%
\subsubsection{Dimension $d=2$}.
%%%%%%%%%
Here the numerical symbol has the form
$$
 F(\tau) = 
\begin{pmatrix}
 a(\tau) & b(\tau)\\
 b(\tau) & -a(\tau)
\end{pmatrix}
$$
with real-valued functions $a$ and $b$. The eigenvalues are
$\lambda_{\pm}(\tau)=\pm\sqrt{a(\tau)^{2}+b(\tau)^{2}}$, implying 
$\Lambda_{-}^{(d)}=-\Lambda_{+}^{(d)}$.
The functions $a$ and $b$ have been studied in the previous examples, $a$ in Example~\ref{Ex:3} and $b$ in Example~\ref{Ex:2}.

In particular, $b(\tau)=0$ for $\tau\in\partial Q$, and therefore
\begin{equation}
\label{E:lambda+=Lambda0}
  \mbox{ for }\:\tau=(\pi,0),\quad \lambda_{+}(\tau) = a(\tau) = \Lambda_{0}
\end{equation}
with the number $\Lambda_{0}=0.5471...$ encountered in Example \ref{Ex:3}, Lemma~\ref{L:im(Fex3)}.
This implies 
$$\Lambda_{+}^{(2)}\ge\Lambda_{0},$$ 
and we are in the same situation as in Example~\ref{Ex:3}:
Strong numerical evidence suggests that the function $\lambda_{+}$ attains its maximum in the point $\tau=(\pi,0)$ and therefore $\Lambda_{+}^{(2)}=\Lambda_{0}$, but we do not have a formal proof for this. The positive eigenvalue $\lambda_{+}$ is plotted in Figure~\ref{F:Ex52D}.
\begin{figure}[h]
\centering
\includegraphics[width=0.6\textwidth]{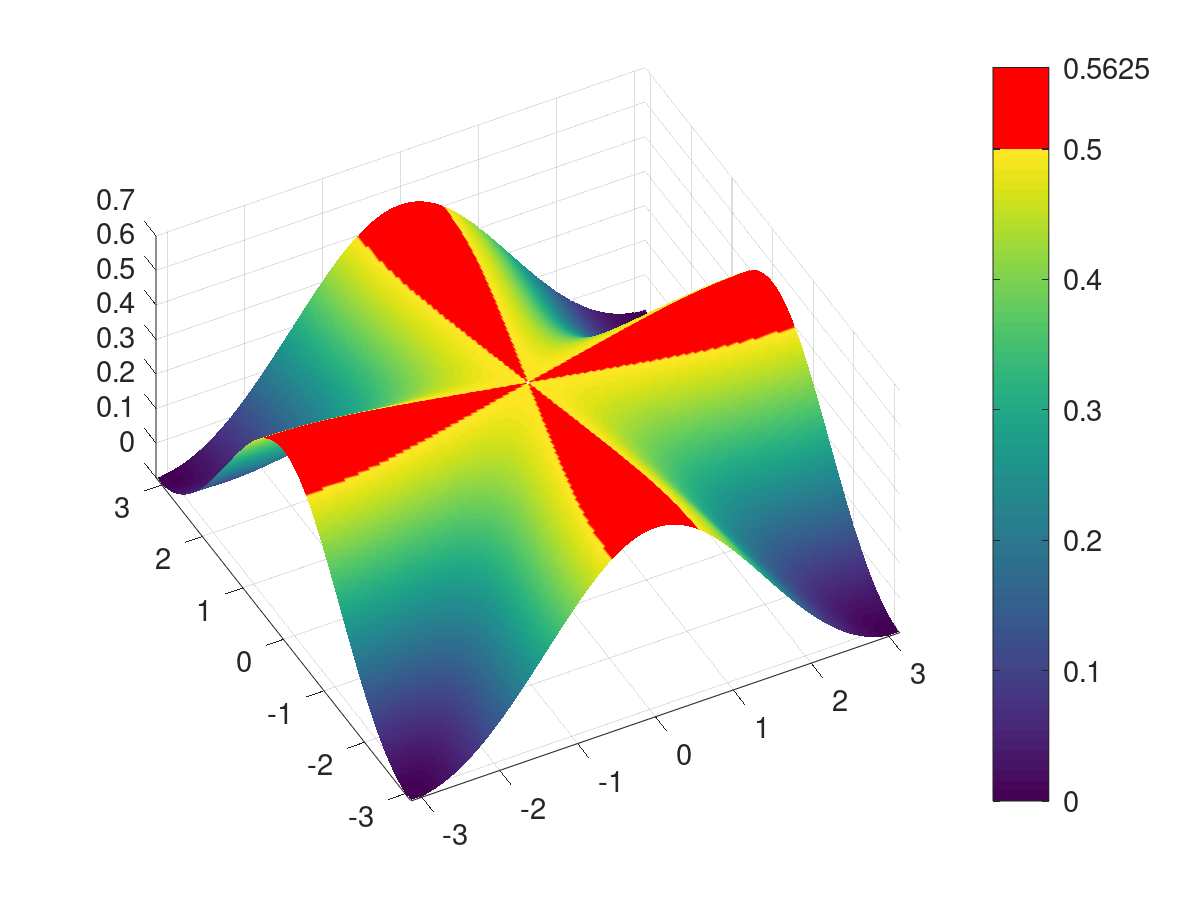}
\caption{$d=2$. Eigenvalue $\lambda_{+}$ on $Q$.}
\label{F:Ex52D}
\end{figure}

In any case, we have proved that in dimension $d=2$ for any bounded open set $\Omega\subset\R^{2}$ and any
$\lambda\in(-0.5471,-0.5)\cup(0.5,0.5471)$ the delta-delta scheme 
$(\lambda\Id-T^{N})U=F$ is unstable in $\ell^{2}(\Omega^{N};\C^{2})$ as $N\to\infty$, whereas the integral equation $(\lambda\Id-A_{\Omega})u=f$ is well posed in $L^{2}(\Omega;\C^{2})$.

%%%%%%%%%%
\subsubsection{Dimension $d=3$}.
%%%%%%%%%
The three eigenvalues $\lambda_{j}$ of $F(\tau)$ satisfy 
$\lambda_{1}+\lambda_{2}+\lambda_{3}=0$.
Numerically, one sees that the minimal and maximal values are attained on the intersection of the boundary of $Q=[-\pi,\pi]^{3}$ with the coordinate planes. 
In Figure~\ref{F:3EVs} we show a graph of the three eigenvalues on the line $\{(\pi,y,0)\mid y\in[-\pi,\pi]\}$. The values $-\frac13$ and $\frac23$ are shown as dashed lines.
\begin{figure}[h]
\centering
\includegraphics[width=0.6\textwidth]{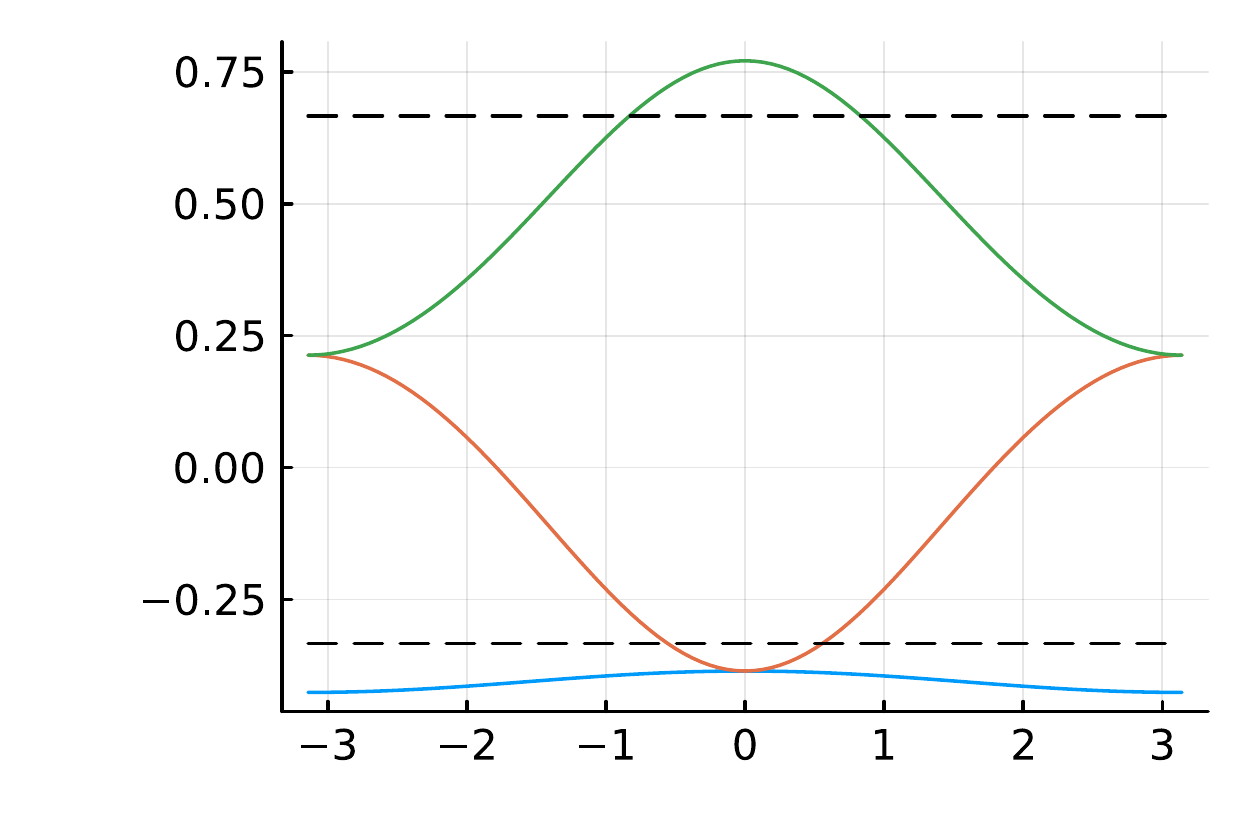}
\caption{$d=3$. Eigenvalues of $F(\tau)$ on middle line of face of $Q$.}
\label{F:3EVs}
\end{figure}

This suggests 
$\Lambda_{-}^{(3)}=\min\Sp(F((\pi,\pi,0)))$ and $\Lambda_{+}^{(3)}=\max\Sp(F((\pi,0,0)))$.

The computed values are 
\begin{equation}
\label{E:Lambda+-3}
\Lambda_{-}^{(3)}=-0.4260241507272727\,,
\qquad
\Lambda_{+}^{(3)}=0.7709022227747195\,.
\end{equation}
This implies a length of $W(T)$ of 
$\Lambda_{+}^{(3)}-\Lambda_{-}^{(3)}=1.1969263735019922$ instead of $1$, which is the length of $W(A)$, thus an overshoot of almost $20\%$.

Under this assumption, one can write simple series expansions for the numbers $\Lambda_{\pm}^{(3)}$. If all the coordinates of $\tau$ are $0$ or $\pi$, then the off-diagonal elements of the matrix $F(\tau)$ vanish and the $3$ eigenvalues are the diagonal elements. Therefore the Fourier series for $F(\tau)$ gives
\begin{equation}
\label{E:Fpm3series}
 \Lambda_{-}^{(3)} = \!\!\!\sum_{m\in\Z^{3},m\ne0} \!\!\!
   \frac{(-1)^{m_{1}+m_{2}}}{4\pi} \frac{m_{1}^{2}+m_{2}^{2}-2m_{3}^{2}}{(m_{1}^{2}+m_{2}^{2}+m_{3}^{2})^{\frac52}}\,,\quad
 \Lambda_{+}^{(3)} = \!\!\!\sum_{m\in\Z^{3},m\ne0} \!\!\!
   \frac{(-1)^{m_{3}}}{4\pi} \frac{m_{1}^{2}+m_{2}^{2}-2m_{3}^{2}}{(m_{1}^{2}+m_{2}^{2}+m_{3}^{2})^{\frac52}}\,.
\end{equation}
These sums, although not absolutely convergent, appear to converge quite well in the sense of partial sums over cubes,
$$
  \sum_{m\in\Z^{3},m\ne0} = \lim_{N\to \infty}\sum_{\max\limits_{j}|m_{j}|\le N,m\ne0} \,.
$$
We do not know whether explicit expressions for these sums exist.

By means of the Clausius-Mossotti relation \eqref{E:ClauMoss} one can express the stability results equivalently in terms of the relative permittivity $\epsilon_{r}$. Let
\begin{equation}
\label{E:epsminmax}
 \epsilon_{\min} = \frac{3\Lambda_{+}-2}{1+3\Lambda_{+}}=0.0943961\dots
 \,,\qquad
 \epsilon_{\max} = \frac{3\Lambda_{-}-2}{1+3\Lambda_{-}}=11.788555\dots\,.
\end{equation}
The numerical range $\lambda\in[-\frac13,\frac23]$ of the quasi-static Maxwell volume integral operator corresponds to $\epsilon_{r}\le0$. The volume integral equation is therefore well posed in $L^{2}(\Omega)$ if the relative permittivity $\epsilon_{r}$ is either non-real or positive.

On the other hand, the corresponding DDA scheme is stable in $\ell^{2}(\Z^{3})$ if and only if $\epsilon_{r}$ is either non-real or contained in the interval 
$(\epsilon_{\min},\epsilon_{max})$. 
For $\epsilon_{r}\in(0,\epsilon_{\min}]\cup[\epsilon_{\max},\infty)$ the integral equation (and therefore the dielectric scattering problem) is well-posed, but the DDA scheme is unstable. 

To conclude this discussion, we show in Table~\ref{T:maxminEVEx5} the result of some computations for the spectrum of the system matrix $T^{N}$ for a cube in three dimensions. 
One can see convergence to the expected values~\eqref{E:Lambda+-3}, even for rather modest values of $N$. Compare also \cite[FIG. 8]{YurkinMinHoekstra2010}.
\begin{table}[htbp]
   \centering
    \begin{tabular}{@{} rccc @{}} % Column formatting, @{} suppresses leading/trailing space
      \toprule
      $N$    &  $\lambda_{\max}(T^{N})$ & $\lambda_{\min}(T^{N})$ & 
                                            $\lambda_{\max}(T^{N})-\lambda_{\min}(T^{N})$\\
      \midrule
4 & 0.67730278666935 & -0.3896455148525014 & 1.06694830152185\\
8 & 0.73653727456221 & -0.4130173055963489 & 1.14955458015856\\
12 & 0.75323748914578 &  -0.4193953119966648 & 1.17263280114245\\
16 & 0.76017444184544 & -0.4220149407429199 & 1.18218938258836\\
      \bottomrule
   \end{tabular}\vglue1ex
   \caption{Computation of $\max(\Sp(T^{N}))$ and $\min(\Sp(T^{N}))$, Example \ref{Ex:5}}
   \label{T:maxminEVEx5}
\end{table}

\subsection*{Acknowledgment}
This work was partially supported by a grant from the \textit{Niels Hendrik Abel Board}.
The authors acknowledge support of the Centre Henri Lebesgue ANR-11-LABX-0020-01.

%%%\section{Les r\'ef\'erences}
%%%%\nocite{*}
\bibliographystyle{acm}
\bibliography{database.bib}                        
 
\end{document}